\documentclass[12pt, reqno, a4paper]{amsart}
\usepackage{times}

\usepackage[utf8]{inputenc}
\usepackage[USenglish]{babel}
\usepackage{amsmath,amsthm,amssymb,amsfonts}
\usepackage{mathtools}
\usepackage{mathrsfs}
\usepackage{bbm}
\usepackage{booktabs}

\usepackage{indentfirst}
\usepackage{enumitem}
\usepackage{xcolor}

\usepackage{float}

\usepackage[breaklinks=true, bookmarksopenlevel=1, bookmarksdepth=2]{hyperref}
\hypersetup{
colorlinks,
   linkcolor={cyan!80!black},
   citecolor={cyan!80!black},
 urlcolor={cyan!80!black}
}

\allowdisplaybreaks

\newtheorem{thm}{Theorem}[section]
\newtheorem{cor}[thm]{Corollary}
\newtheorem{lem}[thm]{Lemma}
\newtheorem{prop}[thm]{Proposition}

\theoremstyle{plain} % just in case the style had changed
\newcommand{\thistheoremname}{}
\newtheorem*{genericthm}{\thistheoremname}

\theoremstyle{definition}
\newtheorem{defn}[thm]{Definition}

\theoremstyle{remark}

\newtheorem*{xrem}{Remark}

\newtheorem*{ntt}{Notation}

\numberwithin{equation}{section}

\newcommand{\N}{\mathbb{N}}      % N = Naturals
\newcommand{\Z}{\mathbb{Z}}      % Z = Integers
\newcommand{\R}{\mathbb{R}}      % R = Reals
\newcommand{\eps}{\varepsilon}   % epsilon

\newcommand{\negphantom}[1]{\settowidth{\dimen0}{$\displaystyle #1$}\hspace*{-\dimen0}}

\renewcommand{\pmod}[1]{         % (mod #1)
  ~(\mathrm{mod}~#1)}

\newcommand{\E}{\mathbb{E}}   % Expectation

\newcommand{\A}{\mathscr{A}}
\renewcommand{\P}{\mathbb{P}}
\newcommand{\li}{\mathrm{li}}

\usepackage[margin=1in]{geometry}

\restylefloat{table}

\setlist[itemize]{leftmargin=*}
\setlist[enumerate]{leftmargin=*}

\begin{document}

\title[Waring and Waring--Goldbach subbases]{Waring and Waring--Goldbach subbases with prescribed representation function}%
\author{Christian T\'afula}%
\address{Instituto de Matem\'atica, Estat\'istica\\
e Ci\^encia da Computa\c{c}\~ao\\
Universidade de S\~ao Paulo\\
Rua do Mat\~ao, 1010\\
S\~ao Paulo, SP 05508-090\\
Brazil}
\curraddr{}
\email{tafula@ime.usp.br}
\thanks{}


% ----------------------------------------------------------------
\begin{abstract}
 Let $h\geq 2$. For $A\subseteq \mathbb{N}$ write
 \[ r_{A,h}(n) := \#\{(x_1,\ldots,x_h)\in A^h ~|~ x_1+\cdots+x_h=n\}. \]
 We prove a general probabilistic subbasis principle: assuming an asymptotic for a weighted $h$-fold representation sum over a basis $B$, there exist subbases $A\subseteq B$ whose representation function $r_{A,h}(n)$ has prescribed regularly varying growth.

 We apply this to $k$-th powers $\mathbb{N}^k$ and to $k$-th powers of primes $\mathbb{P}^k$. For $h \geq k^2-k+O(\sqrt{k})$, we show that every regularly varying function $F$ with $F(x)/\log x\to\infty$ in the admissible range is realized, with the expected singular series factor. In particular, there exists $A\subseteq \mathbb{N}^k$ such that
 \[ r_{A,h}(n)\sim \mathfrak{S}_{k,h}(n) F(n). \]
 Moreover, in the prime setting we obtain thin subbases $A\subseteq \mathbb{P}^k$ with $r_{A,h}(n)\asymp \log n$ for $n$ in the admissible congruence classes.
\end{abstract}

\maketitle
% ----------------------------------------------------------------

%%%%%%%%%%%%%%%%%%%%%%%%%%%%%%%%%%%%%%%%%%%%%%%%%%%%%%%%%
\section{Introduction}
 Let $A \subseteq \N = \mathbb{Z}_{\geq 0}$ and $h \geq 2$ be an integer. Define the representation function  
 \[ r_{A,h}(n) := \#\{(x_1, \ldots, x_h) \in A^h \mid x_1 + \cdots + x_h = n\}, \]  
 which counts the number of ways an integer $n$ can be expressed as a sum of $h$ elements from $A$. In 1956, inspired by a problem of Sidon, Erd\H{o}s \cite{erdo56} used probabilistic methods to show that there exists a subset $A \subseteq \mathbb{N}$ such that $r_{A,2}(n) \asymp \log n$. Later work by Erdős \cite[Corollary 3]{erd86} established that for certain functions $F$ satisfying $\lim_{x\to\infty} F(x)/\log x = \infty$ (e.g., $F(x) = \log x \log\log x$), there exist subsets $A \subseteq \mathbb{N}$ such that $r_{A,2}(n) \sim F(n)$.
 
 In 1990, Erd\H{o}s and Tetali \cite{erdtet90} extended this result for any $h \geq 2$, showing the existence of subsets $A \subseteq \mathbb{N}$ such that $r_{A,h}(n) \asymp \log n$ (so-called \emph{thin bases}). In our earlier work \cite{taf25}, we extended this line of research by studying subsets $A$ whose representation functions grow according to regularly varying functions. Specifically, we proved that for any regularly varying function $F$ such that
 \[ \lim_{x\to\infty} \frac{F(x)}{\log x} = \infty, \qquad F(x) \leq (1 + o(1)) \frac{x^{h-1}}{(h-1)!}, \]
 there exists $A \subseteq \mathbb{N}$ such that $r_{A,h}(n) \sim F(n)$.

 A real-valued function $F : \mathbb{R}_{>0} \to \mathbb{R}_{>0}$ is called \emph{regularly varying} if it is measurable and satisfies $\lim_{x \to \infty} F(\lambda x)/F(x)$ for every $\lambda > 0$. Regularly varying functions are of the form  
 \[ F(x) = x^\kappa \psi(x), \]  
 where $\kappa \in \mathbb{R}$ and $\psi(x)$ is \emph{slowly varying}, meaning $\lim_{x \to \infty} \psi(\lambda x)/\psi(x) = 1$ for any $\lambda > 0$. Slowly varying functions satisfy $\psi(x) = x^{o(1)}$ (see Bingham--Goldie--Teugels \cite{bingham89}). These functions provide a broad and flexible framework for describing growth conditions, encompassing a variety of natural cases.  

 The core of the present paper is a general probabilistic subbasis principle (Theorem \ref{MT2}) that converts suitable weighted $h$-fold additive asymptotics over a basis $B$ into the existence of subbases $A\subseteq B$ whose representation function $r_{A,h}(n)$ has a prescribed regularly varying growth. We then implement this principle for the classical arithmetic sets
 \[ \N^k := \{n^k ~|~ n\in\N\} \qquad\text{and}\qquad \P^k:=\{p^k ~|~ p \text{ prime}\}, \]
 thereby producing subbases of $k$-th powers and of $k$-th powers of primes with prescribed representation functions. In the applications, where we establish weighted asymptotics via the circle method, the admissible range in the number of variables is essentially the sharp one currently obtainable from Vinogradov mean value estimates.

\subsection{Waring subbases}
 In 2000, Vu \cite{vvu00mc} showed that for each \( k \geq 2 \), there exists an integer \( h_k \ll 8^k k^3 \) such that for all \( h \geq h_k \), there exists a subset \( A \subseteq \mathbb{N}^k \) for which \( r_{A,h}(n) \asymp \log n \).\footnote{In \cite[Lemma 2.1]{vvu00wp} it is claimed ``$8^k k^2$'', but the computation of $s_3(k)$ at the start of \cite[Subsection 2.2]{vvu00wp} yields the bound $O(8^k k^3)$.} Later, Wooley \cite{woo03} refined this bound, proving that  
 \[ h_k \leq k\left(\log k + \log \log k + 2 + O\left(\frac{\log \log k}{\log k}\right)\right). \]
 More recently, drawing from the work of Br\"udern--Wooley \cite{bruwoo23}, Pliego \cite{pli24} further improved the bound to 
 \[ h_k \leq \lceil k(\log k + 4.20032) \rceil,\]
 showing that bounds for $h_k$ essentially align with those for the asymptotic order of $\N^k$ as derived from the circle method. Pliego also established that for certain functions $\psi(x)=x^{o(1)}$ with $\lim_{x\to\infty}\psi(x)/\log x=\infty$, there exist $A\subseteq \N^k$ such that
 \[ r_{A,h}(n) \sim \mathfrak{S}_{k,h}(n) \psi(n), \]
 where \( \mathfrak{S}_{k,h}(n) \) denotes the singular series associated with Waring’s problem, defined by
 \begin{equation}
  S(q,a) := \sum_{r=1}^{q} e\bigg(\frac{ar^{k}}{q}\bigg),\qquad \mathfrak{S}_{k,h}(n) := \sum_{q\geq 1}\sum_{\substack{a=1 \\ (a,q)=1}}^{q} \frac{S(q,a)^h}{q^h} e\bigg(-\frac{na}{q}\bigg). \label{singser}
 \end{equation}
 Our results differ in two ways. First, we address a prescribed growth problem within regular variation, covering a broad range $\log x \ll F(x) \ll x^{h/k-1}$. Second, our general Theorem \ref{MT2} applies beyond $\N^k$, and in particular yields the analogous results for $\P^k$.
 
 Our first theorem treats $\N^k$. Let $H_0(k)$ be defined by
 \begin{equation}\label{h0k1}
  H_0(k) := \begin{cases}
          2^{k-1}, & 1\leq k\leq 4,\\
          \frac{1}{2}k(k-1)+\lfloor\sqrt{2k+2}\rfloor, & k\geq 5,
         \end{cases} 
 \end{equation}
 so that Hua's inequality in the form of Wooley \cite[Corollary 14.7]{woo19} implies
 \[ \int_0^1 \bigg|\sum_{n\leq x} e(\alpha n^k)\bigg|^{s}\,\mathrm{d}\alpha \ll x^{s-k+o(1)} \qquad (s\geq 2H_0(k)). \]
 (See Lemma \ref{war-hua}). %Note that $2H_0(k) + 1= k^2 - k + O(\sqrt{k})$.

 \begin{thm}\label{MT1}
  Let $k\geq 2$ and $h\geq 2H_0(k)+1$. Let $F$ be a regularly varying function satisfying
  \[ \lim_{x\to\infty}\frac{F(x)}{\log x}=\infty,\qquad
     F(x)\leq (1+o(1))\frac{\Gamma(1+1/k)^h}{\Gamma(h/k)}x^{h/k-1}. \]
  Then there exists $A\subseteq \N^k$ such that $r_{A,h}(n)\sim \mathfrak{S}_{k,h}(n) F(n)$.
  
  Moreover, if $\log x \ll F(x) \ll x^{h/k-1}$, then there exists $A\subseteq \N^k$ with $r_{A,h}(n) \asymp F(n)$.
 \end{thm}
 
 Theorem \ref{MT1} is obtained by combining our abstract subbasis theorem (Theorem \ref{MT2}) with a weighted asymptotic for $r_{\N^k,h}(n)$ (a variant of a theorem of Wooley \cite{woo04}, with the admissible range of $h$ sharpened using the above form of Hua's inequality).

\subsection{Waring--Goldbach subbases}
 We also treat the prime analogue $\P^k$. Here congruence obstructions intervene, as is standard in Waring--Goldbach problems. Let $k\geq 1$, and given a prime $p$, let $\theta = \theta(k,p)$ be the unique integer such that $p^{\theta} \mid k$ but $p^{\theta+1}\nmid k$. Define
 \begin{equation} \gamma = \gamma(k,p) := \begin{cases}
                             \theta + 2 &\text{ if } p=2,\, 2\mid k, \\
                             \theta + 1 &\text{ otherwise},
                            \end{cases}
    \qquad K(k) := \prod_{(p-1)\mid k} p^{\gamma}. \label{Kkval}
 \end{equation}
 If a prime $q$ is coprime to $K(k)$, then $q^{k} \equiv 1 \pmod{p^{\gamma}}$ whenever $(p-1)\mid k$, since $\varphi(p^{\gamma}) = p^{\theta}(p-1) \mid k$ for $p$ odd (and $\varphi(2^{\gamma})/2 \mid k$), where $\varphi(n):= \#\{m\leq n ~|~ (m,n)=1 \}$ is Euler's totient function. So by the Chinese remainder theorem, if $n$ is a sum of $h$ $k$-th powers of primes larger than $k+1$, then necessarily $n\equiv h\pmod{K(k)}$. If $k$ is odd, then $K(k) =2$.
 
 \begin{table}[h!]
  \centering
  \begin{tabular}{ccccccccccc}
  \toprule
  $k$     & 1 & 2 & 3 & 4 & 5 & 6 & 7 & 8 & 9 & 10 \\ \midrule
  $K(k)$  & 2 & 24 & 2 & 240 & 2 & 504 & 2 & 480 & 2 & 264 \\ \bottomrule
  \end{tabular}
  \caption{First few values of $K(k)$.}
  \label{tab:kk}
 \end{table}

 \noindent
 The corresponding singular series associated to Waring--Goldbach's problem is defined as
 \begin{equation}
  S^{*}(q,a) := \sum_{\substack{r=1 \\ (r,q)=1}}^{q} e\bigg(\frac{ar^{k}}{q}\bigg),\qquad \mathfrak{S}^{*}_{k,h}(n) := \sum_{q\geq 1} \sum_{\substack{a=1 \\ (a,q)=1}}^{q} \frac{S^{*}(q,a)^h}{\varphi(q)^h}e\bigg(-\frac{na}{q}\bigg). \label{singserp}
 \end{equation}
 
 Our next result is based on a weighted Waring--Goldbach asymptotic, serving as the prime analogue of Wooley's theorem for $\N^k$ \cite[Theorem 1.1]{woo04}. Let $H_0(k)$ be as in \eqref{h0k1}. In particular, the hypothesis $h\geq 2H_0(k)+1$ again amounts to $h \geq k^2 - k + O(\sqrt{k})$.
 
 \begin{thm}\label{MTp}
  Let $k\geq 1$ and $h\geq 2H_0(k)+1$, and let $\delta>0$ be any real number with 
  \[ \delta < \frac{h-2H_0(k)}{2 h(h-1) H_0(k)}. \]
  Then, for any $\omega \geq 1/h - \delta$,
  \begin{equation*}
   \begin{aligned}
    \sum_{\substack{x_1,\ldots,x_h \in \P^k \\ x_1+\cdots+x_h = N}} (x_1\cdots x_h)^{\omega-\frac{1}{k}}(\log x_1 \cdots \log x_h) = \mathfrak{S}^{*}_{k,h}(N)\,\frac{\Gamma(\omega)^h}{\Gamma(h\omega)} \, N^{h\omega-1} + O_R\bigg(\frac{N^{h\omega -1}}{(\log N)^{R}}\bigg)
   \end{aligned}
  \end{equation*}
  for every $R>1$.
 \end{thm}
 %Note that the implied constant also depends on $h$, $k$, but we omit this dependency.
 Combining this input with the general subbasis principle (Theorem \ref{MT2}), we obtain subbases $A\subseteq \P^k$ with prescribed growth, including thin subbases. 

 \begin{thm}\label{MTp1}
  Let $k\geq 1$ and $h\geq 2H_0(k)+1$. Let $F$ be a regularly varying function satisfying
  \[ \lim_{x\to\infty}\frac{F(x)}{\log x}=\infty,\qquad
     F(x)\leq (1+o(1))\frac{\Gamma(1/k)^h}{\Gamma(h/k)}\frac{x^{h/k-1}}{(\log x)^h}. \]
  Then there exists $A\subseteq \P^k$ such that $r_{A,h}(n)\sim \mathfrak{S}^{*}_{k,h}(n) F(n)$ \textnormal{($n\equiv h\pmod{K(k)}$)}.
  
  Moreover, if $\log x \ll F(x) \ll x^{h/k-1}/(\log x)^h$, then there exists $A\subseteq \P^k$ with $r_{A,h}(n) \asymp F(n)$ \textnormal{($n\equiv h\pmod{K(k)}$)}.
 \end{thm}
 
 Theorem \ref{MTp1} implies, in particular, that there exists $A\subseteq \P$ such that $r_{A,3}(n) \asymp \log n$ for $n$ odd, there exists $B\subseteq \P^2$ such that $r_{B,5}(n) \asymp \log n$ for $n\equiv 5\pmod{24}$, and so on.
 
 \begin{xrem}
  Wirsing \cite[Theorem 2]{wir86} showed that for every $h\geq 3$, there exists a subset $A\subseteq \P$ with $|A\cap [1,x]| \ll (x \log x)^{1/h}$ such that every $n\equiv h\pmod{2}$ can be written as a sum of $h$ elements of $A$. Our result strengthens this by producing $B\subseteq \P$ with $|B\cap [1,x]| \asymp (x \log x)^{1/h}$ (by \eqref{szofAA}) such that $r_{B,h}(n) \asymp \log n$ for $n\equiv h\pmod{2}$.

  Assuming the asymptotic version of Goldbach's conjecture, Granville \cite[Theorem 2]{gra07} showed that there is a subset $A\subseteq \P$ with $|A\cap [1,x]| \asymp (x\log x)^{1/2}$ such that every large even integer is the sum of two elements of $A$. Our result, however, does not apply to this case.
 \end{xrem}

%%%%%%%%%%%%%%%%%%%%%%%%%%%%%%%%%%%%%%%%%%%%%%%%%%%%%%%%%
\subsection{General subbases}
 Theorems \ref{MT1} and \ref{MTp1} will be derived from the general Theorem \ref{MT2}, which abstracts the probabilistic construction and isolates the analytic input needed in applications. Let $B\subseteq \N$ and write $B(x):=|B\cap[1,x]|$. Assume that
 \begin{equation}\label{oregvar}
  B(x)\asymp x^{\beta}\vartheta(x)
 \end{equation}
 for some $\beta>0$ and some slowly varying $\vartheta$. In addition, suppose that $B$ satisfies a mean value estimate of Hua type:
 \begin{equation}\label{hua-type}
  \sum_{n\leq x} r_{B,H_0}(n)^2 \ll \frac{B(x)^{2H_0}}{x}\,x^{o(1)}
 \end{equation}
 for some $H_0=H_0(B)\geq 1$. This condition asks for $hB$ to have nearly minimal additive energy, in the sense that $\sum_{n\leq x} r_{B,h}(n)^2 \geq (\sum_{n\leq x} r_{B,h}(n))^2/x \asymp B(x)^{2h}/x$ by Cauchy--Schwarz. It will only be used in the proof of Theorem \ref{nonexact}. 
 
 \begin{thm}\label{MT2}
  Suppose $B\subseteq \N$ satisfies \eqref{oregvar} and \eqref{hua-type}, and let $h \geq 2H_0+1$. Let $F(x)=x^{\kappa}\psi(x)$ be a regularly varying function with $F(x)\leq (1+o(1))\,x^{h\beta-1}\vartheta(x)^h$, and define $f(x) := (xF(x))^{1/h}$. Suppose that
  \begin{equation}
   \sum_{\substack{x_1,\ldots, x_h\in B \\ x_1+\cdots+ x_h = n}} \frac{f(x_1)}{x_1^{\beta}\vartheta(x_1)}\cdots \frac{f(x_h)}{x_h^{\beta}\vartheta(x_h)} \sim \mathfrak{S}(n)F(n), \label{mainest}
  \end{equation}
  where $\mathfrak{S}(n)$ is some function of $n$ (depending on $h,B,F$) satisfying 
  \[ \mathfrak{S}(n)\asymp 1 \quad \text{for } n\in \mathscr{S}, \]
  for some subset $\mathscr{S}\subseteq \Z_{\geq 0}$. Then:
  \begin{enumerate}[label=\textnormal{(\roman*)}]
   \item If $F(n)/\log n \to \infty$, there exists $A\subseteq B$ such that 
   \[ r_{A,h}(n) \sim \mathfrak{S}(n)F(n) \quad (n\in\mathscr{S}). \]
  
   \item If $F(n)\gg \log n$, there exists $A\subseteq B$ such that 
   \[ r_{A,h}(n) \asymp F(n) \quad (n\in\mathscr{S}). \]
  \end{enumerate}
 \end{thm}

 In our applications, the subset $\mathscr{S}$ in Theorem \ref{MT2} is either all the positive integers (for $k$-th powers) or the admissible congruence classes $a\cdot \N+b$ (for $k$-th powers of primes). To apply Theorem \ref{MT2} it is often enough to establish \eqref{mainest} in the power case $F(x)=x^{\kappa}$; we record two corollaries which extend this to a wide class of slowly varying perturbations.
 
 The first corollary shows that many standard perturbations (e.g. powers of $\log$) can be treated directly, though in general one obtains $\asymp$-bounds rather than an asymptotic formula.
 
 \begin{cor}\label{cor-tech}
  Let $B$ and $h$ be as in Theorem \ref{MT2}, and suppose \eqref{mainest} holds with $F(x)=x^{\kappa}$ for some $0\leq \kappa \leq h\beta-1$. Let $\mathcal{L}$ be a measurable positive function such that $\mathcal{L}(x^{\lambda}) \asymp_{\lambda} \mathcal{L}(x)$ for every $\lambda>0$, and write $G(x) := x^{\kappa}\mathcal{L}(x)$. If $\log x \ll G(x) \ll x^{h\beta-1}\vartheta(x)^h$, then there exists $A\subseteq B$ such that
  \[ r_{A,h}(n)\asymp G(n) \quad (n\in\mathscr{S}). \]
 \end{cor}
 
 \begin{xrem}
  The condition imposed on $\mathcal{L}$ is stronger than slow variation. Nevertheless, it is still compatible with slow variation in the sense that for any such $\mathcal{L}$ one has $\mathcal{L}\asymp \psi$ for some slowly varying $\psi$ (see the proof of Corollary \ref{cor-tech}). Moreover, the argument for Corollary \ref{cor-tech} yields an asymptotic formula under the stronger hypothesis $\mathcal{L}(x)\sim \mathcal{L}(x^{\lambda})$ for every $\lambda>0$ (for instance $\mathcal{L}(x)=\log\log x$), producing $A\subseteq B$ with $r_{A,h}(n)\sim \mathfrak{S}(n)\,G(n)$ ($n\in \mathscr{S}$).
 \end{xrem}

 The second corollary says that, for a fixed exponent $\kappa$, to obtain the full family of slowly varying perturbations it suffices to supplement the power case with a suitable upper bound at a slightly smaller exponent.
 
 \begin{cor}\label{cor-tech2}
  Let $B$ and $h$ be as in Theorem \ref{MT2}, and suppose that \eqref{mainest} holds with $F(x)=x^{\kappa}$ for some $0\leq \kappa \leq h\beta-1$. Assume that there exists $\delta>0$ such that, with $\omega=(1+\kappa)/h$,
  \begin{equation*}
   \sum_{\substack{x_1,\ldots, x_h\in B \\ x_1+\cdots+ x_h = n}} \frac{x_1^{\omega-\delta}}{x_1^{\beta}\vartheta(x_1)}\cdots \frac{x_h^{\omega-\delta}}{x_h^{\beta}\vartheta(x_h)} \ll n^{h(\omega-\delta)-1}.
  \end{equation*}
  Then for every slowly varying function $\psi$, the asymptotic \eqref{mainest} holds with $F$ replaced by $G(x):=x^{\kappa}\psi(x)$; that is,
  \begin{equation*}
   \sum_{\substack{x_1,\ldots, x_h\in B \\ x_1+\cdots+ x_h = n}} \frac{g(x_1)}{x_1^{\beta}\vartheta(x_1)}\cdots \frac{g(x_h)}{x_h^{\beta}\vartheta(x_h)} \sim \mathfrak{S}(n)\, G(n) \quad (n\in\mathscr{S}),
  \end{equation*}
  where $g(x):=(xG(x))^{1/h}$. In particular:
  \begin{enumerate}[label=\textnormal{(\roman*)}]
   \item If
   \[ \lim_{x\to\infty}\frac{G(x)}{\log x}=\infty,\qquad G(x) \leq (1+o(1))\, x^{h\beta-1}\vartheta(x)^h, \]
   then there exists $A\subseteq B$ such that $r_{A,h}(n)\sim \mathfrak{S}(n)\,G(n)$ \textnormal{($n\in\mathscr{S}$)}.\medskip
   
   \item If $\log x\ll G(x)\ll x^{h\beta-1}\vartheta(x)^h$, then there exists $A\subseteq B$ with $r_{A,h}(n)\asymp G(n)$ \textnormal{($n\in\mathscr{S}$)}.
  \end{enumerate}
 \end{cor}
 
 In combination with Lemma \ref{szofA}, the proof of Theorem \ref{MT2} also yields the asymptotic density of the resulting subbases if $B(x)$ is regularly varying: for $F(x)=x^{\kappa}\psi(x)$ one has
 \begin{equation}\label{szofAA}
  |A\cap[1,x]|\sim \frac{h\beta}{1+\kappa}(xF(x))^{1/h}.
 \end{equation}
 
 \begin{ntt}
  Throughout the paper, we use the common asymptotic notation $\sim$, $o$, $O$, $\gg$, $\ll$, $\asymp$ with a subscript to indicate dependency on some parameter (omitting dependencies on $h$, $k$, $f$, $F$, and $B$). For a sequence of random variables $(X_n)_{n\geq 1}$ in a probability space $\Omega$, and a real function $g$, we write $X_{n} \stackrel{\text{a.s.}}{\sim} g(n)$ if $X_n(\omega)\sim g(n)$ for almost all $\omega \in\Omega$. Also, $e(\alpha) := e^{2\pi i \alpha}$.
 \end{ntt}

 %%%%%%%%%%%%%%%%%%%%%%%%%%%%%%%% 
\section{Bounds for weighted solution counts}\label{sec-ne}
 Suppose that $B\subseteq \N$ satisfies \eqref{oregvar} and \eqref{hua-type}. The main goal of this section is to establish the following general estimate for weighted counts of solutions to linear equations with variables restricted to $B$.

 \begin{thm}\label{nonexact}
  Let $h\geq 2H_0+1$ and $1\leq \ell \leq h-1$. Suppose $f(x)=x^{\omega}\phi(x)$ is a regularly varying function with $\omega \geq 1/h$, and let $c_1,\ldots,c_{\ell}\in \Z_{\geq 1}$ with $c_1+\cdots+c_{\ell}\leq C$. Then, for any $0<\eta<(h-\ell)(\omega-\frac{1}{h}) + \tfrac{1}{h(h-1)}$,
  \[ \sum_{\substack{x_1,\ldots, x_{\ell}\in B \\ c_1 x_1+\cdots+ c_{\ell}x_{\ell} = N}} \frac{f(x_1)}{B(x_1)} \cdots \frac{f(x_{\ell})}{B(x_{\ell})} \ll_C N^{-\eta}\, \frac{f(N)^{h}}{N}. \]
 \end{thm}

 Theorem \ref{nonexact} shows that contribution of solutions to $x_1+\cdots+x_h = N$ with repeated (e.g., $x_1=x_2$) or small variables is negligible in weighted counts. In particular, it allows us to replace $f$ and $B(x)$ by asymptotic equivalents in sums over solutions to $x_1+\cdots+x_h=N$, as the following corollary makes precise. 
 
 \begin{cor}\label{cor1}
  Let $h\geq 2H_0+1$, and let $f(x)=x^{\omega}\phi(x)$ be a regularly varying function with $\omega \geq 1/h$. For any $0<\eta< \tfrac{1}{h(h-1)}$, we have
  \[ \sum_{\substack{x_1,\ldots, x_{h}\in B \\ x_1+\cdots+x_h = N}} \frac{f(x_1)}{x_1^{\beta}\vartheta(x_1)} \cdots \frac{f(x_h)}{x_h^{\beta}\vartheta(x_h)} = \sum_{\substack{x_1,\ldots, x_h\in B \\ x_1+\cdots+ x_h = N \\ \forall j,\, x_j\geq N^{\eta/2\omega}}} \frac{f(x_1)}{x_1^{\beta}\vartheta(x_1)} \cdots \frac{f(x_h)}{x_h^{\beta}\vartheta(x_h)} + O\bigg(N^{-\frac{\eta}{2}+o(1)} \frac{f(N)^h}{N}\bigg). \]
 \end{cor}
 \begin{proof}
  Since $f(x)/x^{\beta}\vartheta(x) \asymp f(x)/B(x)$, Theorem \ref{nonexact} with $\ell=h-1$ gives
  \begin{align*}
   \sum_{\substack{x_1,\ldots, x_h\in B \\ x_1+\cdots+ x_h = N \\ \exists j ~|~ x_j < N^{\eta/2\omega}}} \frac{f(x_1)}{x_1^{\beta}\vartheta(x_1)} \cdots \frac{f(x_h)}{x_h^{\beta}\vartheta(x_h)} &\ll \sum_{\substack{n \leq N^{\eta/2\omega} \\ n\in B}} \frac{f(n)}{B(n)} \sum_{\substack{x_1,\ldots, x_{h-1}\in B \\ x_1+\cdots+ x_{h-1} = N-n}} \frac{f(x_1)}{B(x_1)}\cdots \frac{f(x_{h-1})}{B(x_{h-1})} \\
   &\ll N^{-\eta} \frac{f(N)^h}{N} \sum_{\substack{n \leq B(N^{\eta/2\omega})}} \frac{f(b_n)}{n},
  \end{align*}
  where $b_n$ is the $n$-th element of $B$. Since $B(x) = x^{\beta+o(1)}$, we have $b_n = n^{1/\beta + o(1)}$, hence
  \[ \sum_{\substack{n \leq B(N^{\eta/2\omega})}} \frac{f(b_n)}{n} = \sum_{\substack{n \leq B(N^{\eta/2\omega})}} n^{\omega/\beta - 1 + o(1)} = N^{\frac{\eta}{2} + o(1)}.  \]
  Combined with the previous line, this completes the proof.
 \end{proof}

\subsection{Auxiliary lemmas}
 The first result we need is the Carleson--Hunt theorem, which bounds the maximal operator of Fourier partial sums on $L^p$. In our setting, it lets us replace a supremum over truncated sums by the full sum up to a constant factor.
 
 \begin{lem}[Carleson--Hunt theorem]\label{carlhun}
  Let $1<p<\infty$ and let $(a_n)_{1\leq n \leq N}$ be a finite sequence of complex numbers. Then there exists a constant $C_p>0$, depending only on $p$, such that
  \[ \int_0^1 \bigg(\sup_{1\leq M\leq N} \bigg|\sum_{n=1}^{M} a_n e(n\alpha)\bigg|\bigg)^p \,\mathrm{d}\alpha \leq C_p \int_0^1 \bigg|\sum_{n=1}^{N} a_n e(n\alpha)\bigg|^p \,\mathrm{d}\alpha. \]
 \end{lem}
 \begin{proof}
  See Theorem 11.2.1 (b), p. 456 of Grafakos \cite{grafakos09}.
 \end{proof}
  
 To control exponential sums with coefficients arising from regularly varying functions, it is often convenient to replace a given function by a piecewise-constant approximation on dyadic intervals. %This suppresses oscillations of the slowly varying factor and ensures that the resulting function behaves in a controlled, monotone manner.
 We refer to this modification as the \emph{dyadic smoothing} of the function.

 \begin{defn}\label{diasmo}
  Let $\phi$ be a slowly varying function. We define the \emph{dyadic smoothing} $\widetilde{\phi}$ of $\phi$ by
  \[ \widetilde{\phi}(x) := \sup_{\substack{2^{J} \leq y < 2^{J+1} \\ J = \lfloor \frac{\log x}{\log 2}\rfloor}} \phi(y). \]
  If $f(x) = x^{\omega}\phi(x)$ is regularly varying, we set $\widetilde{f}(x) := x^{\omega}\widetilde{\phi}(x)$.
 \end{defn}

 For slowly varying functions $\phi$, the uniform convergence theorem (cf. BGT \cite[Theorem 1.2.1]{bingham89}) implies
 \begin{equation}
  \lim_{x \to \infty} \sup_{\lambda \in [\frac{1}{2}, 2]} \frac{\phi(\lambda x)}{\phi(x)} = 1. \label{uniconv}
 \end{equation}
 It then follows immediately from the definition of dyadic smoothing that $f(x) \sim \widetilde{f}(x)$. In particular, one may replace $f$ by $\widetilde{f}$ in estimates without affecting the main term, while benefiting from the fact that $\widetilde{f}$ is essentially monotone on dyadic intervals.

 \begin{lem}\label{lemtch}
  Let $f(x) = x^{\omega}\phi(x)$ be regularly varying. Then
  \[ \sum_{x \leq n \leq 2x} |\widetilde{f}(n+1) - \widetilde{f}(n)| \ll_\omega f(x). \]
 \end{lem}
 \begin{proof}
  Write $\widetilde{f}(n) = n^\omega \widetilde{\phi}(n)$. By definition, $\widetilde{\phi}$ is constant on each dyadic interval $[2^J,2^{J+1})$, so there is at most one integer $n^* \in [x,2x]$ where $\widetilde{\phi}(n^*) \neq \widetilde{\phi}(n^*+1)$. For $n \in [x,2x]\setminus\{n^*\}$, by \eqref{uniconv} we have
  \[ |\widetilde{f}(n+1)-\widetilde{f}(n)| = \widetilde{\phi}(n)\,|(n+1)^\omega - n^\omega| \sim \phi(x)\,|(n+1)^\omega - n^\omega|. \]
  Summing over $n \in [x,2x]$ gives
  \[ \sum_{x \leq n \leq 2x} |\widetilde{f}(n+1)-\widetilde{f}(n)| \ll \phi(x) \sum_{x \leq n \leq 2x} |(n+1)^\omega - n^\omega| + |\widetilde{f}(n^*+1)-\widetilde{f}(n^*)|. \]
  Since $n \mapsto n^\omega$ is monotone, $\sum_{x\leq n \leq 2x} |(n+1)^\omega - n^\omega| \ll_\omega x^\omega$, and the remaining jump at $n^*$ satisfies $|\widetilde{f}(n^*+1)-\widetilde{f}(n^*)| \ll f(x)$. Combining these estimates and using $x^\omega \phi(x) = f(x)$ completes the proof.
 \end{proof}

%%%%%%%%%%%%%%%%%%%%%%%
\subsection{Weighted mean value bounds}
 By Parseval’s identity, it is a basic consequence of \eqref{hua-type} that for any $h\geq 2H_0$,
 \begin{align}
  \int_{0}^{1} \bigg|\sum_{n\leq x} \mathbbm{1}_{B}(n)\, e(n\alpha)\bigg|^{h}\,\mathrm{d}\alpha 
  &\leq B(x)^{h-2H_0} \int_{0}^{1} \bigg|\sum_{n\leq x} \mathbbm{1}_{B}(n)\, e(n\alpha)\bigg|^{2H_0}\,\mathrm{d}\alpha \nonumber \\
  &\leq B(x)^{h-2H_0} \sum_{n\leq H_0x} r_{B,H_0}(n)^2 \nonumber \\
  &\ll \frac{B(x)^h}{x} \, x^{o(1)}. \label{hua-type2}
 \end{align}
 The following lemma is a weighted version of this estimate.

 \begin{lem}\label{weighted-hua}
  Let $f(x) = x^{\omega}\phi(x)$ be a regularly varying function. Then, for every $\ell\geq 1$,
  \[ \int_{0}^{1} \bigg|\sum_{n\leq x} \widetilde{f}(n)\,\frac{\mathbbm{1}_{B}(n)}{n^{\beta}} \, e(n\alpha)\bigg|^{\ell} \,\mathrm{d}\alpha \ll x^{\max\{\ell\omega - \frac{\ell}{\ell^{*}},\,0\} + o(1)}, \]
  where $\ell^{*} := \max\{\ell,2H_0\}$ and $\widetilde{f}$ is as in Definition \ref{diasmo}.
 \end{lem}

 In particular, for $\ell = h\geq 2H_0$ and $\omega \geq 1/h$, we obtain $x^{h\omega-1+o(1)}$ in the upper bound.

 \begin{proof}%[Proof of Lemma \ref{weighted-hua}]
  Write $g(\alpha;x) := \sum_{n\leq x} \mathbbm{1}_{B}(n)e(n\alpha)$. By applying partial summation (for integer $x$), and using Lemma \ref{lemtch} to control the differences of $\widetilde{f}(n)/n^{\beta} = \widetilde{(f(n)/n^{\beta})}$, we obtain
  \begin{align*}
   \bigg|\sum_{n\leq x} \widetilde{f}(n) &\,\frac{\mathbbm{1}_{B}(n)}{n^{\beta}}\, e(n\alpha)\bigg| \\
   &\leq \frac{\widetilde{f}(x)}{x^{\beta}} |g(\alpha;x)| + \sum_{n<x} |g(\alpha;n)| \bigg|\frac{\widetilde{f}(n+1)}{(n+1)^{\beta}} - \frac{\widetilde{f}(n)}{n^{\beta}} \bigg| \\
   &\leq \frac{\widetilde{f}(x)}{x^{\beta}} |g(\alpha;x)| + \sum_{j\leq \frac{\log x}{\log 2}} \bigg(\sup_{\frac{x}{2^{j+1}} \leq y< \frac{x}{2^j}} |g(\alpha;y)| \bigg) \sum_{\frac{x}{2^{j+1}}\leq n< \frac{x}{2^{j}}} \bigg|\frac{\widetilde{f}(n+1)}{(n+1)^{\beta}} - \frac{\widetilde{f}(n)}{n^{\beta}} \bigg| \\
   &\ll \sum_{j\leq \frac{\log x}{\log 2}} \bigg(\sup_{\frac{x}{2^{j+1}} \leq y< \frac{x}{2^j}} |g(\alpha;y)| \bigg) \frac{f(2^{-j} x)}{(2^{-j} x)^{\beta}}
  \end{align*}
  Thus, by the triangle inequality,\footnote{We use that $\lVert \sum_{j=1}^{M} f_j\rVert_{\ell} \leq \sum_{j=1}^{M} \lVert f_j \rVert_{\ell} \leq M\max_{1\leq j\leq M} \lVert f_j \rVert_{\ell}$.} we find
  \begin{align}
   \int_{0}^{1} \bigg|\sum_{n\leq x} \widetilde{f}(n) \frac{\mathbbm{1}_{B}(n)}{n^{\beta}} e(n\alpha)\bigg|^{\ell}\,\mathrm{d}\alpha &\ll \int_{0}^{1} \Bigg(\sum_{j\leq \frac{\log x}{\log 2}} \bigg(\sup_{\frac{x}{2^{j+1}} \leq y< \frac{x}{2^j}} |g(\alpha;y)| \bigg) \frac{f(2^{-j} x)}{(2^{-j} x)^{\beta}}\Bigg)^{\ell}\,\mathrm{d}\alpha \nonumber \\
   &\ll (\log x)^{\ell}\,\max_{j\leq \frac{\log x}{\log 2}} \frac{f(2^{-j} x)^{\ell}}{(2^{-j} x)^{\ell\beta}} \bigg(\int_{0}^{1} \sup_{\frac{x}{2^{j+1}} \leq y< \frac{x}{2^j}} |g(\alpha;y)|^{\ell}\,\mathrm{d}\alpha \bigg). \nonumber %\label{logmax}
  \end{align}
  
  By monotonicity of $L^p$ norms (since $\ell^{*} \geq \ell$), applying \eqref{hua-type2} (since $\ell^{*} \geq 2H_0 >1$), the Carleson--Hunt theorem (Lemma \ref{carlhun}) yields
  \begin{align*}
   \int_{0}^{1} \sup_{\frac{x}{2^{j+1}} \leq y< \frac{x}{2^j}} |g(\alpha;y)|^{\ell}\,\mathrm{d}\alpha &\leq \bigg(\int_{0}^{1} \sup_{\frac{x}{2^{j+1}} \leq y< \frac{x}{2^j}} |g(\alpha;y)|^{\ell^{*}}\,\mathrm{d}\alpha\bigg)^{\ell/\ell^{*}} \\
   &\ll \bigg(\int_{0}^{1} |g(\alpha; 2^{-j}x)|^{\ell^{*}}\,\mathrm{d}\alpha\bigg)^{\ell/\ell^{*}} \\
   &\ll \frac{B(2^{-j}x)^{\ell}}{(2^{-j}x)^{\ell/\ell^{*}}}(2^{-j}x)^{o(1)} = (2^{-j}x)^{\ell\beta - \frac{\ell}{\ell^{*}} + o(1)}.
  \end{align*}
  Consequently, since $f(x) = x^{\omega+o(1)}$, we have
  \begin{align*}
   \int_{0}^{1} \bigg|\sum_{n\leq x} \widetilde{f}(n) \frac{\mathbbm{1}_{B}(n)}{n^{\beta}} e(n\alpha)\bigg|^{\ell}\,\mathrm{d}\alpha &\ll (\log x)^{\ell} \bigg(\max_{j\leq \frac{\log x}{\log 2}} \frac{f(2^{-j}x)^{\ell}}{(2^{-j}x)^{\ell/\ell^{*}}} (2^{-j}x)^{o(1)}\bigg) \\
   &= x^{o(1)} \bigg(\max_{j\leq \frac{\log x}{\log 2}} (2^{-j}x)^{\ell\omega- \frac{\ell}{\ell^{*}} + o(1)}\bigg) \\
   &= x^{\max\{\ell\omega - \frac{\ell}{\ell^{*}},\,0\} + o(1)}. \qedhere
  \end{align*}
 \end{proof}
 
 Lemma \ref{weighted-hua} gives a general bound for the full initial segment. When the sum is restricted to a dyadic interval, we get a better estimate.
 
 \begin{lem}\label{dyadic-hua}
  Under the assumptions of Lemma \ref{weighted-hua}, for any $C \geq 1$ one has
  \[ \int_{0}^{1} \bigg|\sum_{x/C \leq n\leq x} \widetilde{f}(n)\,\frac{\mathbbm{1}_{B}(n)}{n^{\beta}} \, e(n\alpha)\bigg|^{\ell} \,\mathrm{d}\alpha \ll_C x^{\ell\omega - \frac{\ell}{\ell^{*}} + o(1)}. \]
 \end{lem}
 \begin{proof}
  Set again $g(\alpha;x) := \sum_{n\leq x} \mathbbm{1}_{B}(n)e(n\alpha)$. Applying partial summation, together with Lemma \ref{lemtch}, gives
  \begin{align*}
   \bigg|\sum_{x/C \leq n\leq x} \widetilde{f}(n) &\,\frac{\mathbbm{1}_{B}(n)}{n^{\beta}}\, e(n\alpha)\bigg| \\
   &\ll_C \bigg(\sup_{x/C \leq y\leq x} |g(\alpha;y)| \bigg)\bigg( \frac{\widetilde{f}(x)}{x^{\beta}} + \sum_{x/C\leq n<x} \bigg|\frac{\widetilde{f}(n+1)}{(n+1)^{\beta}} - \frac{\widetilde{f}(n)}{n^{\beta}} \bigg|\bigg) \\
   &\ll_C \bigg(\sup_{x/C \leq y\leq x} |g(\alpha;y)| \bigg) \frac{f(x)}{x^{\beta}}
  \end{align*}
  As $\ell^{*}\geq \ell$, the monotonicity of $L^p$ norms (or H\"older) yields
  \begin{align}
   \int_{0}^{1} \bigg|\sum_{x/C \leq n\leq x} \widetilde{f}(n) \frac{\mathbbm{1}_{B}(n)}{n^{\beta}} e(n\alpha)\bigg|^{\ell}\,\mathrm{d}\alpha &\ll_C \frac{f(x)^{\ell}}{x^{\ell\beta}} \int_{0}^{1} \sup_{x/C \leq y\leq x} |g(\alpha;y)|^{\ell} \,\mathrm{d}\alpha \nonumber \\
   &\leq \frac{f(x)^{\ell}}{x^{\ell\beta}} \Bigg(\int_{0}^{1} \sup_{x/C \leq y\leq x} |g(\alpha;y)|^{\ell^{*}} \,\mathrm{d}\alpha\Bigg)^{\ell/\ell^{*}}. \nonumber %\label{logmax2}
  \end{align}
  By the Carleson--Hunt theorem (Lemma \ref{carlhun}) together with \eqref{hua-type2},
  \begin{align*}
   \int_{0}^{1} \sup_{x/C \leq y< x} |g(\alpha;y)|^{\ell^{*}}\,\mathrm{d}\alpha \ll \int_{0}^{1} |g(\alpha; x)|^{\ell^{*}}\,\mathrm{d}\alpha &\ll x^{\ell^{*}\beta - 1 + o(1)}.
  \end{align*}
  Substituting this estimate and recalling that $f(x)=x^{\omega+o(1)}$ completes the proof:
  \begin{align*}
   \int_{0}^{1} \bigg|\sum_{x/C \leq n\leq x} \widetilde{f}(n) \frac{\mathbbm{1}_{B}(n)}{n^{\beta}} e(n\alpha)\bigg|^{\ell}\,\mathrm{d}\alpha &\ll_C \frac{f(x)^{\ell}}{x^{\ell\beta}} (x^{\ell^{*}\beta-1 + o(1)})^{\frac{\ell}{\ell^{*}}} = x^{\ell\omega - \frac{\ell}{\ell^{*}}+o(1)}. \qedhere
  \end{align*}
 \end{proof}

 We are now ready to prove the theorem.
 
 \subsection{Proof of Theorem \ref{nonexact}}
 Let $h\geq 2H_0+1$, and $1\leq \ell \leq h-1$. First observe that
 \begin{equation*}
  \sum_{\substack{x_1,\ldots, x_{\ell}\in B \\ c_1 x_1+\cdots+ c_{\ell} x_{\ell} = N}} \frac{f(x_1)}{B(x_1)}\cdots \frac{f(x_{\ell})}{B(x_{\ell})} \asymp \sum_{\substack{x_1,\ldots, x_{\ell}\in B \\ c_1 x_1+\cdots+ c_{\ell} x_{\ell} = N}} \frac{g(x_1)}{x_1^{\beta}}\cdots \frac{g(x_{\ell})}{x_{\ell}^{\beta}},
 \end{equation*}
 where $g(x) := \widetilde{f}(x)/\widetilde{\vartheta}(x)$. Since the ratio of slowly varying functions is slowly varying, $g$ is regularly varying. Note also that $g = \widetilde{g}$ by definition, and $g(x) = x^{\omega+o(1)}$.
 
 To handle the linear relation $c_1x_1+\cdots+c_\ell x_\ell=N$, introduce the exponential sums
 \[ T(\alpha; x) := \sum_{n\leq x} g(n)\,\frac{\mathbbm{1}_{B}(n)}{n^{\beta}}\, e(\alpha n),\qquad T^{\sharp}(\alpha; x) := \sum_{x/C \leq n\leq x} g(n)\,\frac{\mathbbm{1}_{B}(n)}{n^{\beta}}\, e(\alpha n). \]
 Every solution to $c_1x_1+\cdots+c_\ell x_\ell=N$ has at least one variable $\geq N/C$, so
 \[ \sum_{\substack{x_1,\ldots, x_{\ell}\in B \\ c_1 x_1+\cdots+ c_{\ell} x_{\ell} = N}} \frac{g(x_1)}{x_1^{\beta}}\cdots \frac{g(x_{\ell})}{x_{\ell}^{\beta}} \leq \sum_{\substack{I\subseteq \{1,\ldots,\ell\} \\ I\neq\varnothing}} \sum_{\substack{x_1,\ldots, x_{\ell}\in B \\ c_1 x_1+\cdots+ c_{\ell} x_{\ell} = N \\ \forall i\in I,\, x_i\geq N/C}} \frac{g(x_1)}{x_1^{\beta}}\cdots \frac{g(x_{\ell})}{x_{\ell}^{\beta}} \]

 By orthogonality and H\"older’s inequality,
 \begin{align}
  \sum_{\substack{x_1,\ldots, x_{\ell}\in B \\ c_1 x_1+\cdots+ c_{\ell} x_{\ell} = N \\ \forall i\in I,\, x_i\geq N/C}} \frac{g(x_1)}{x_1^{\beta}}\cdots &\frac{g(x_{\ell})}{x_{\ell}^{\beta}} = \int_{0}^{1} \Bigg(\prod_{i\in I} T^{\sharp}(c_i\alpha; N)\Bigg)\Bigg(\prod_{\substack{j=1 \\ j\notin I}}^{\ell} T(c_{j}\alpha; N)\Bigg) e(-N\alpha)\,\mathrm{d}\alpha \nonumber \\
  &\leq \prod_{i\in I} \bigg(\int_{0}^{1} |T^{\sharp}(c_i\alpha; N)|^{\ell}\,\mathrm{d}\alpha \bigg)^{1/\ell} \prod_{\substack{j=1\\ j\notin I}}^{\ell} \bigg(\int_{0}^{1} |T(c_{j}\alpha; N)|^{\ell}\,\mathrm{d}\alpha \bigg)^{1/\ell} \nonumber \\
  &\ll_C \bigg(\int_{0}^{1} |T^{\sharp}(\alpha; N)|^{\ell}\,\mathrm{d}\alpha\bigg)^{|I|/\ell} \bigg(\int_{0}^{1} |T(\alpha; N)|^{\ell}\,\mathrm{d}\alpha\bigg)^{1-|I|/\ell}. \nonumber
 \end{align}
 Applying Lemmas \ref{weighted-hua} and \ref{dyadic-hua} (recall $\ell^{*}=\max\{\ell,2H_0\}$), we obtain
 \begin{align*}
  \sum_{\substack{x_1,\ldots, x_{\ell}\in B \\ c_1 x_1+\cdots+ c_{\ell} x_{\ell} = N}} \frac{g(x_1)}{x_1^{\beta}}\cdots \frac{g(x_{\ell})}{x_{\ell}^{\beta}} &\ll_C \sum_{j=1}^{\ell} \bigg(\int_{0}^{1} |T^{\sharp}(\alpha; N)|^{\ell}\,\mathrm{d}\alpha\bigg)^{j/\ell} \bigg(\int_{0}^{1} |T(\alpha; N)|^{\ell}\,\mathrm{d}\alpha\bigg)^{1-j/\ell} \\
  &\ll_{C} \sum_{j=1}^{\ell} \big(N^{\ell\omega - \frac{\ell}{\ell^{*}} + o(1)} \big)^{j/\ell}\, \big(N^{\max\{\ell\omega - \frac{\ell}{\ell^{*}},\,0\} + o(1)} \big)^{1-j/\ell} \\
  &= \sum_{j=1}^{\ell} N^{\max\{\ell(\omega-\frac{1}{\ell^{*}}),\,j(\omega-\frac{1}{\ell^{*}})\} + o(1)} \\
  &= N^{\max\{\ell(\omega-\frac{1}{\ell^{*}}),\,\omega-\frac{1}{\ell^{*}}\} + o(1)}.
 \end{align*}
  
 For $j=1$ or $\ell$, since $\ell^{*} \leq h-1$ and $\omega \geq 1/h$, the exponent satisfies  
 \begin{align*}
  j(\omega-\tfrac{1}{\ell^{*}}) &= h\omega- 1 - ((h-j)(\omega-\tfrac{1}{h}) + j(\tfrac{1}{\ell^{*}}-\tfrac{1}{h})), \\
  &\leq h\omega- 1 - ((h-\ell)(\omega-\tfrac{1}{h}) + \tfrac{1}{h(h-1)}).
 \end{align*}
 Since $f(N)^h/N = N^{h\omega-1+o(1)}$, this shows that the contribution is at most $f(N)^h/N$ with a power saving
 \[ N^{-\eta}, \qquad \eta := (h-\ell)(\omega-\tfrac{1}{h}) + \tfrac{1}{h(h-1)} - \eps \]  
 for any $\eps>0$, completing the proof. \hfill$\square$

%%%%%%%%%%%%%%%%%%%%%%%%%%%
\section{General subbases: Theorem \ref{MT2}}\label{rdss}
 We continue to assume that \( B \subseteq \mathbb{N} \) satisfies \eqref{oregvar} and \eqref{hua-type}. Let \( h \geq 2H_0 + 1 \) be an integer, and let $\log x \ll F(x) \ll x^{h\beta-1}\vartheta(x)^h$ be a regularly varying function satisfying \eqref{mainest}. Define
 \[ f(x) := (xF(x))^{1/h} = x^{\omega}\phi(x). \]
 We now construct a probability space whose elements are subsets \( \mathscr{A} \subseteq B \). Each integer \( n \in B \) is included independently with probability
 \begin{equation}
  \Pr(\mathbbm{1}_{\mathscr{A}}(n) = 1) = \mathbb{E}(\mathbbm{1}_{\mathscr{A}}(n)) := \min\bigg\{ c\frac{f(n)}{n^{\beta}\vartheta(n)},\,1 \bigg\}\mathbbm{1}_B(n) \qquad (\forall n \geq 1), \label{presp}
 \end{equation}
 where \( c>0 \) is a constant to be specified later so that \( c f(n) \leq (1+o(1))\,n^{\beta}\vartheta(n) \).\footnote{Hence, in particular, $\min\{cf(n), n^{\beta}\vartheta(n)\} \sim cf(n)$, which is the relevant property of the chosen weights.} The variables \( \mathbbm{1}_{\mathscr{A}}(n) \) are taken to be mutually independent $\{0,1\}$-random variables.

 When the stronger condition \( B(x) \sim x^{\beta}\vartheta(x) \) holds in place of \eqref{oregvar}, the random set \( \mathscr{A} \) has a counting function that almost surely matches \( f \) asymptotically:

 \begin{lem}\label{szofA}
  If \( B(x)\sim x^{\beta}\vartheta(x) \), then $|\mathscr{A} \cap [1, x]| \stackrel{\textnormal{a.s.}}{\sim} c\,\dfrac{\beta}{\omega}\, f(x)$.
 \end{lem}
 \begin{proof}
  Since \( B(x) \) is regularly varying, its asymptotic inverse $b_{\lfloor x \rfloor} := \inf\{y \in \mathbb{R}_{\geq 0} : B(y) \geq \lfloor x \rfloor\}$ is also regularly varying, and satisfies \( b_k^{\beta}\vartheta(b_k) \sim k \). By the strong law of large numbers,
  \[ |\mathscr{A} \cap [1, x]| \stackrel{\text{a.s.}}{\sim} c \sum_{n \leq x} \frac{f(n)}{n^{\beta}\vartheta(n)} \mathbbm{1}_B(n) \sim c \sum_{k \leq B(x)} \frac{f(b_k)}{k}. \]
  Because the composition of regularly varying functions is regularly varying, we can express $f(b_{\lfloor x \rfloor})/\lfloor x \rfloor = x^{\omega / \beta - 1} \xi(x)$ for some slowly varying function \( \xi(x) \). Hence
  \begin{align*}
   |\mathscr{A} \cap [1, x]| &\stackrel{\text{a.s.}}{\sim} c \int_{1}^{B(x)} t^{\omega / \beta - 1} \xi(t)\, \mathrm{d}t \\
   &\sim c \bigg(\int_{1 / B(x)}^{1} u^{\omega / \beta - 1} \frac{\xi(u B(x))}{\xi(B(x))}\, \mathrm{d}u \bigg) f(x).
  \end{align*}

  By Potter bounds (see BGT \cite[Theorem 1.5.6 (i)]{bingham89}), for every \( \delta > 0 \) there exists \( C = C_\delta > 0 \) such that, for all large \( x \geq x_\delta \) and \( C/x \leq u \leq 1 \), we have $\xi(ux)/\xi(x) \leq 2u^{-\delta}$. Choosing \( \delta < \omega / 2\beta \), we decompose $\int_{1 / B(x)}^{1} = \int_{1 / B(x)}^{C / B(x)} + \int_{C / B(x)}^{1}$. Since \( \xi \) is slowly varying,
  \[ u^{\omega / \beta - 1} \frac{\xi(u B(x))}{\xi(B(x))}\,\mathbbm{1}_{(C / B(x), 1]} \xrightarrow[x \to \infty]{} u^{\omega / \beta - 1} \mathbbm{1}_{(0,1]}, \]
  and the dominated convergence theorem gives
  \[ \int_{C / B(x)}^{1} u^{\omega / \beta - 1} \frac{\xi(u B(x))}{\xi(B(x))}\, \mathrm{d}u \to \int_{0}^{1} u^{\omega / \beta - 1}\, \mathrm{d}u = \frac{\beta}{\omega}. \]
  Finally, since \( \xi(x) = x^{o(1)} \), the contribution of \( \int_{1 / B(x)}^{C / B(x)} \) vanishes, completing the proof.
 \end{proof}

%%%%%%
\subsection{Expectation of \texorpdfstring{$r_{\A,h}(n)$}{r\_A,h(n)}}\label{rrho}
 We now turn to the computation of the expectation of  
 \[ r_{\A,h}(n) = \sum_{\substack{x_1, \ldots, x_h \in \mathbb{Z}_{\geq 0} \\ x_1 + \cdots + x_h = n}} \mathbbm{1}_{\A}(x_1) \cdots \mathbbm{1}_{\A}(x_h). \]
 To make this quantity more amenable to probabilistic analysis, we introduce a family of auxiliary representation functions. Let \( 1 \leq \ell \leq h \) and \( c_1, \ldots, c_{\ell} \in \mathbb{Z}_{\geq 1} \). A solution \( (x_1, \ldots, x_{\ell}) \) to the equation \( c_1 x_1 + \cdots + c_{\ell} x_{\ell} = n \) is called \emph{exact} if the variables \( x_i \) are pairwise distinct. The corresponding \emph{exact representation function} (associated with \( c_1, \ldots, c_{\ell} \)) is defined by  
 \[ \rho_{\A, \ell}^{(c_1, \ldots, c_{\ell})}(n)
 := \sum_{\substack{x_1, \ldots, x_{\ell} \in \mathbb{Z}_{\geq 0} \\ c_1 x_1 + \cdots + c_{\ell} x_{\ell} = n \\ x_i \text{s distinct}}} \mathbbm{1}_{\A}(x_1) \cdots \mathbbm{1}_{\A}(x_{\ell}). \]
 When \( c_1 = \cdots = c_{\ell} = 1 \), we simply write \( \rho_{\A, \ell}(n) \).

 To illustrate the connection with \( r_{\A,h}(n) \), consider a solution \( (x_1, \ldots, x_h) \) to \( x_1 + \cdots + x_h = n \) where \( x_1 = x_2 \) but the remaining variables \( x_2, \ldots, x_h \) are distinct.  
 In this case, the tuple \( (x_2, \ldots, x_h) \) forms an exact solution to  
 \[ 2x_2 + x_3 + \cdots + x_h = n. \]
 More generally, every non-exact solution of length \( h \) corresponds to an exact solution of an equation with fewer variables, where the coefficients sum to \( h \).  
 This observation yields the decomposition
 \begin{equation}
  r_{\A, h}(n) = \rho_{\A, h}(n) + \sideset{}{'}\sum_{\substack{(c_1, \ldots, c_{\ell})}}
   \rho_{\A, \ell}^{(c_1, \ldots, c_{\ell})}(n), \label{decomprho}
 \end{equation}
 where the sum runs over \( 1 \leq \ell \leq h - 1 \) and all tuples \( (c_1, \ldots, c_{\ell}) \in \mathbb{Z}_{\geq 1}^{\ell} \) satisfying \( c_1 + \cdots + c_{\ell} = h \). By applying Theorem \ref{nonexact} to control the contribution of non-exact solutions, we get that the main term in the expectation of \( r_{\A,h}(n) \) arises from the exact solutions.
 
 \begin{lem}\label{mainlm}
  For any $0<\eta< (\omega-\tfrac{1}{h}) + \tfrac{1}{h(h-1)}$, we have
  \[ \E(r_{\A,h}(n)) = \E(\rho_{\A,h}(n)) + O(n^{-\eta} F(n)). \]
  Moreover,
  \[ \E(\rho_{\A,h}(n)) \sim c^{h} \mathfrak{S}(n) F(n) \qquad (n \in \mathscr{S}). \]
 \end{lem}
 \begin{proof}
  By the decomposition \eqref{decomprho},
  \[ \E(r_{\A,h}(n))
     = \E(\rho_{\A,h}(n))
       + \sideset{}{'}\sum_{\substack{(c_1,\ldots,c_{\ell})}}
         \E(\rho_{\A,\ell}^{(c_1,\ldots,c_{\ell})}(n)). \]
  For each tuple $(c_1, \ldots, c_{\ell})$, Theorem~\ref{nonexact} implies
  \[ \E(\rho_{\A,\ell}^{(c_1,\ldots,c_{\ell})}(n))
     \ll c^{h}
       \sum_{\substack{x_1, \ldots, x_{\ell} \in B \\ c_1x_1 + \cdots + c_{\ell}x_{\ell} = n}}
       \frac{f(x_1)}{x_1^{\beta}\vartheta(x_1)} \cdots \frac{f(x_{\ell})}{x_{\ell}^{\beta}\vartheta(x_{\ell})}
     \ll n^{-\eta} F(n), \]
  proving the first part. 

  For the second, by Corollary \ref{cor1}, we may choose $0<\eta< \tfrac{1}{h(h-1)}$ such that
  \begin{align*}
   \E(\rho_{\A,h}(n))
   &= \sum_{\substack{x_1, \ldots, x_h \in B \\ x_1 + \cdots + x_h = n \\ x_i \text{ distinct}}}
      \min\bigg\{ c\frac{f(x_1)}{x_1^{\beta}\vartheta(x_1)},\,1 \bigg\}
      \cdots
      \min\bigg\{ c\frac{f(x_h)}{x_h^{\beta}\vartheta(x_h)},\,1 \bigg\} \\
   &= (1 + o(1))\, c^{h}
      \sum_{\substack{x_1, \ldots, x_h \in B \\
                      x_1 + \cdots + x_h = n \\
                      \forall j,\, x_j \ge n^{\eta / 2\omega}}}
        \frac{f(x_1)}{x_1^{\beta}\vartheta(x_1)} \cdots \frac{f(x_h)}{x_h^{\beta}\vartheta(x_h)} \\
   &\hspace{4em}
     + O\Bigg(
        n^{-\frac{\eta}{2} + o(1)} \frac{f(n)^{h}}{n}
        + c^{h}
          \sideset{}{'}\sum_{\substack{(c_1,\ldots,c_{\ell})}}
          \sum_{\substack{x_1, \ldots, x_{\ell} \in B \\
                          c_1x_1 + \cdots + c_{\ell}x_{\ell} = n}}
            \frac{f(x_1)}{x_1^{\beta}\vartheta(x_1)} \cdots \frac{f(x_{\ell})}{x_{\ell}^{\beta}\vartheta(x_{\ell)}} \Bigg).
  \end{align*}
  By Theorem \ref{nonexact}, the error term is $O(n^{-\frac{\eta}{2}+o(1)} F(n))$, while \eqref{mainest} yields the main term
  \[ \E(\rho_{\A,h}(n)) = (1 + o(1))\, c^{h} \mathfrak{S}(n) F(n), \]
  completing the proof.
 \end{proof}
 
 Using Lemma \ref{mainlm}, the proof of Theorem \ref{MT2} reduces to showing that, in the probability space defined at the beginning of this section, we have
 \[ r_{\A,h}(n) \stackrel{\textnormal{a.s.}}{\sim} \E(r_{\A,h}(n)). \]
 Since an event of probability one is non-empty, this implies the existence of a subset \( A \subseteq B \) such that \( r_{A,h}(n) \sim \E(r_{\A,h}(n)) \).

 Before proceeding, we prove a lemma that will be used to handle lower-order terms arising in the concentration estimates.

 \begin{lem}\label{psav}
  Let \( 1 \leq \ell \leq h-1 \) and let \( (c_1,\ldots,c_{\ell}) \in \Z_{\ge 1}^{\ell} \)
  satisfy \( c_1+\cdots+c_{\ell} \leq C \).
  Then, for any \( 0 < \eta < (\omega - \tfrac{1}{h}) + \tfrac{1}{h(h-1)} \), we have
  \[ \E(r^{(c_1,\ldots,c_{\ell})}_{\A,\ell}(n)) \ll_{C} n^{-\eta} F(n). \]
 \end{lem}
 \begin{proof}
  As in the decomposition \eqref{decomprho}, we have
  \[ \E(r^{(c_1,\ldots,c_{\ell})}_{\A,\ell}(n)) \leq \E(\rho^{(c_1,\ldots,c_{\ell})}_{\A,\ell}(n)) + \sideset{}{''}\sum_{\substack{(d_1,\ldots,d_{t})}} \E(\rho^{(d_1,\ldots,d_{t})}_{\A,t}(n)), \]
  where the sum runs over \( 1 \leq t \leq \ell-1 \) and tuples \( (d_1,\ldots,d_{t}) \in \Z_{\geq 1}^{t} \) with \( d_1+\cdots+d_{t} \leq C \). As in the proof of Lemma \ref{mainlm}, Theorem \ref{nonexact} implies that \(\E(r^{(c_1,\ldots,c_{\ell})}_{\A,\ell}(n)) \ll n^{-\eta} F(n)\).
 \end{proof}

 We now turn to the two essential concentration lemmas.

\subsection{Concentration of boolean polynomials}
 Let \( n \geq 1 \), and consider independent (not necessarily identically distributed) random variables \( v_1, \ldots, v_n \), each taking values in \( \{0,1\} \). A \emph{boolean polynomial} is a multivariate polynomial of the form
 \[ Y(v_1, \ldots, v_n) = \sum_{i} c_i I_i \in \mathbb{R}[v_1, \ldots, v_n], \]
 where each \( I_i \) is a monomial in the variables \( v_k \). We classify such polynomials according to the following properties:
 \begin{itemize}
  \item \emph{Positive}: if \( c_i > 0 \) for all \( i \);
  \item \emph{Simple}: if the largest exponent of any \( v_k \) in any monomial is at most \( 1 \);
  \item \emph{Homogeneous}: if all monomials have the same total degree;
  \item \emph{Normal}: if \( 0 \leq c_i \leq 1 \) for all \( i \), and the constant term of \( Y \) is zero.
 \end{itemize}

 For a non-empty multiset\footnote{A \emph{multiset} is a collection in which elements may occur more than once.} \( S \subseteq \{v_1, \ldots, v_n\} \), define
 \[ \partial_S := \prod_{v \in S} \partial_{v}, \]
 where \( \partial_v \) denotes differentiation with respect to \( v \). For example, if \( S = \{1,1,2\} \), then
 \[ \partial_S(v_1^3 v_2 v_3 + 3v_1^5) = 6v_1v_3. \]
 We further define the quantities
 \[ \E_j(Y) := \max_{\substack{S \subseteq \{v_1, \ldots, v_n\} \\ \text{multiset},\, |S| = j}} \E(\partial_S Y), \qquad \text{and} \qquad \E'(Y) := \max_{j \geq 1} \E_j(Y). \]

 Two key concentration results for boolean polynomials will be applied:

 \begin{thm}[Kim--Vu \cite{kimvvu00}]\label{kimvu}
  Let \( d \geq 1 \), and suppose \( Y(v_1, \ldots, v_n) \) is a positive, simple boolean polynomial of degree \( d \). Define \( E' := \E'(Y) \) and \( E := \max\{\E(Y), E'\} \). Then, for any \( \lambda \geq 1 \),
  \[ \Pr\big(|Y - \E(Y)| > 8^d \sqrt{d!}\,\lambda^{d - 1/2}(E'E)^{1/2}\big) \ll_d n^{d-1} e^{-\lambda}. \]
 \end{thm}

 In applications, we take \( \lambda = (d+1)\log n \). This result is most effective when \( 1 \ll \E' \ll \E / (\log n)^{2d} \). The next theorem addresses the complementary regime, where \( \E \) is small.

 \begin{thm}[Vu {\cite[Theorem 1.4]{vvu00mc}}]\label{vuORI}
  Let \( d \geq 2 \), and let \( Y(v_1, \ldots, v_n) = \sum_i c_i I_i \) be a simple, homogeneous, normal boolean polynomial of degree \( d \). For any \( \alpha, \gamma > 0 \), there exists a constant \( \Lambda = \Lambda(d, \alpha, \gamma) \) such that if \( \E_1(Y), \ldots, \E_{d-1}(Y) \leq n^{-\alpha} \), then for any \( 0 < \lambda \leq \E(Y) \),
  \[ \Pr(|Y - \E(Y)| \geq (\lambda \E(Y))^{1/2}) \leq 2d\, e^{-\lambda / 16d\Lambda} + n^{-\gamma}. \]
 \end{thm}

 \begin{xrem}
  The random variable \( r_{\A,h}(n) \) can be viewed as a boolean polynomial:
  \[ r_{\A,h}(n) = Y(v_1, \ldots, v_n), \]
  where the \( v_i \) are independent \(\{0,1\}\)-random variables with \( \Pr(v_i = 1) = \E(\mathbbm{1}_{\A}(i)) \). If \( S = \{v_{x_1}, \ldots, v_{x_{\ell}}\} \) is a multiset of indices with \( x_1 + \cdots + x_{\ell} = k \), then 
  \[ \E(\partial_S r_{\A,h}(n)) \leq \frac{h!}{(h-\ell)!}\, \E(r_{\A,h-\ell}(n - k)). \]
 \end{xrem}

\subsection{Theorem \ref{MT2}: Case \texorpdfstring{$\kappa>0$}{kappa>0}}
 Let $c=1$. We apply Kim--Vu (Theorem \ref{kimvu}) to the boolean polynomial $\rho_{\A,h}(n)$. Writing
 \[ E = \E(\rho_{\A,h}(n)),\qquad E' = \E'(\rho_{\A,h}(n))\ll 1+\max_{1\leq \ell\leq h-1}\max_{k\leq n}\E(r_{\A,\ell}(k)), \]
 and using Lemmas \ref{mainlm} and \ref{psav}, we obtain the bounds
 \[ E \asymp F(n)=n^{\kappa+o(1)},\qquad E'\ll n^{\max\{0,\kappa-\eta\}+o(1)}. \]
 Take $\lambda=(h+1)\log n$. Then
 \begin{align*}
  \lambda^{h-\frac{1}{2}}(E'E)^{1/2}
  &< (h+1)^h(\log n)^h \,\bigg(\frac{E'}{E}\bigg)^{1/2} E \\
  &\ll (\log n)^h\bigg(\frac{n^{\max\{0,\kappa-\eta\}+o(1)}}{n^{\kappa+o(1)}}\bigg)^{1/2} E
  = n^{-\min\{\kappa,\eta\}/2+o(1)} E.
 \end{align*}
 Hence Kim--Vu gives
 \[ \Pr\Big(|\rho_{\A,h}(n)-\E(\rho_{\A,h}(n))| > C_h\,n^{-\min\{\kappa,\eta\}/2+o(1)} E\Big) \ll n^{-2}, \]
 for some constant $C_h>0$ depending only on $h$. Since the right-hand side has a convergent series over $n$, the Borel--Cantelli lemma implies
 \[ \rho_{\A,h}(n)\stackrel{\text{a.s.}}{\sim}\E(\rho_{\A,h}(n))\sim c^h\mathfrak{S}(n)F(n). \]

 It remains to show that every non-exact contribution is a.s. negligible. By \eqref{decomprho} it suffices to prove that for each $1\leq \ell\leq h-1$ and each tuple $(c_1,\dots,c_\ell)$ with $c_1+\cdots+c_\ell\leq h$,
 \[ \rho_{\A,\ell}^{(c_1,\ldots,c_{\ell})}(n) \stackrel{\text{a.s.}}{=} o(F(n)). \]
 Applying Kim--Vu to the boolean polynomial $\rho_{\A,\ell}^{(c_1,\ldots,c_{\ell})}(n)$, Lemma \ref{psav} gives
 \[ E,\;E'\ll 1+n^{-\eta}F(n)=n^{\max\{0,\kappa-\eta\}+o(1)}. \]
 With the same choice $\lambda=(h+1)\log n,$ this yields
 \[ \Pr\Big(|\rho_{\A,\ell}^{(c_1,\ldots,c_{\ell})}(n)-\E(\rho_{\A,\ell}^{(c_1,\ldots,c_{\ell})}(n))| > D_h\, n^{\max\{0,\kappa-\eta\}+o(1)}\Big)\ll n^{-2}, \]
 for some constant $D_h>0$ depending only on $h$. Using Borel--Cantelli and Lemma \ref{psav}, which gives $\E(\rho_{\A,\ell}^{(c_1,\ldots,c_{\ell})}(n))\ll n^{-\eta}F(n)$, we find
 \[ \rho_{\A,\ell}^{(c_1,\ldots,c_{\ell})}(n) \stackrel{\text{a.s.}}{\ll} \E(\rho_{\A,\ell}^{(c_1,\ldots,c_{\ell})}(n)) + n^{\max\{0,\kappa-\eta\}+o(1)} = o(F(n)), \]
 which completes the proof.\hfill$\square$\medskip

 We now focus on the case $\kappa=0$.

%%%%%%%%%%%%%%%%%%%%%%%%%
\subsection{\texorpdfstring{$\delta$}{delta}-small and \texorpdfstring{$\delta$}{delta}-normal solutions}
 Following Vu \cite{vvu00wp}, to estimate \( r_{\A,h}(n) \), we divide the contributing solutions into two types according to a parameter \( \delta \). For \( 0 < \delta < 1 \), define
 \begin{equation}\label{deltans}
  \begin{aligned}
   r^{(\delta\textnormal{-small})}_{\A,h}(n)
   &:= \sum_{\substack{x_1,\ldots,x_h \in \N \\ x_1+\cdots+x_h = n \\ \exists j : x_j < n^{\delta}}}
       \mathbbm{1}_{\A}(x_1)\cdots\mathbbm{1}_{\A}(x_h), \\
   r^{(\delta\textnormal{-normal})}_{\A,h}(n)
   &:= \sum_{\substack{x_1,\ldots,x_h \in \N \\ x_1+\cdots+x_h = n \\ \forall j,~x_j \geq n^{\delta}}}
       \mathbbm{1}_{\A}(x_1)\cdots\mathbbm{1}_{\A}(x_h).
  \end{aligned}
 \end{equation}
 Thus,
 \[ r_{\A,h}(n) = r^{(\delta\textnormal{-small})}_{\A,h}(n) + r^{(\delta\textnormal{-normal})}_{\A,h}(n). \]
 Analogous quantities \( \rho^{(\delta\textnormal{-small})}_{\A,h}(n) \) and \( \rho^{(\delta\textnormal{-normal})}_{\A,h}(n) \) are defined similarly. We shall show that the contribution of \(\delta\)-small solutions is negligible on average.

 \begin{lem}\label{expdsml}
  Let \( 0 < \eta < (\omega - \tfrac{1}{h}) + \tfrac{1}{h(h-1)} \). For every \( 0 < \delta < \eta \), we have
  \[ \E(r^{(\delta\textnormal{-small})}_{\A,h}(n)) \ll n^{\omega\delta - \eta} F(n). \]
 \end{lem}
 \begin{proof}
  We begin by bounding the expectation:
  \begin{align*}
   \E(r^{(\delta\textnormal{-small})}_{\A,h}(n)) &\leq \sum_{k \leq n^{\delta}} \E(r_{\A,h-1}(n-k)\,\mathbbm{1}_{\A}(k)) \\
   &\leq c \sum_{k \leq n^{\delta}} \frac{f(k)}{k^{\beta}\vartheta(k)}\,\mathbbm{1}_{B}(k)\, \E(r_{\A,h-1}(n-k)\mid \mathbbm{1}_{\A}(k)=1).
  \end{align*}
  For \( k \leq n^{\delta} \), Lemma~\ref{psav} yields
  \[ \E(r_{\A,h-1}(n-k)\mid \mathbbm{1}_{\A}(k)=1) \leq \sum_{\ell=1}^{h-1}\E(r_{\A,h-\ell}(n-\ell k)) \ll n^{-\eta}F(n). \]
  Substituting back gives
  \[ \E(r^{(\delta\textnormal{-small})}_{\A,h}(n)) \ll n^{-\eta}F(n) \sum_{k \leq n^{\delta}} \frac{f(k)}{k^{\beta}\vartheta(k)}\,\mathbbm{1}_{B}(k). \]
  Since \(f(k)\) is regularly varying and \( k^{\beta}\vartheta(k) \asymp B(k)\), the sum is asymptotically proportional to \(f(n^{\delta})\) (by the same argument as in Lemma \ref{szofA}). Hence
  \[ \E(r^{(\delta\textnormal{-small})}_{\A,h}(n)) \ll n^{-\eta}F(n)\, f(n^{\delta}) = n^{\omega\delta - \eta + o(1)} F(n). \]
  Since the range of $\eta$ is open and $\delta<\eta$, take $0<\eps<\eta-\delta$ and replace $\eta$ by $\eta-\eps$; for large $n$, the $o(1)$ term is then absorbed.
 \end{proof}

%%%%%%%%%%%%%%%%%%%%
\subsection{Disjoint families of representations}
 Let $\widehat{r}_{\A,\ell}(n)$ denote the maximum size of a disjoint family of representations $R = (x_1,\ldots,x_{\ell}) \in \A^\ell$ satisfying $x_1+\cdots+x_{\ell}= n$. In other words, $\widehat{r}_{\A,\ell}(n) = |\mathcal{M}|$, where $\mathcal{M}$ is a maximal disjoint family (abbreviated \emph{maxdisfam}) of such representations. By definition, for every representation $R$, there exists $S\in \mathcal{M}$ such that $S\cap R\neq \varnothing$. This yields the inequality
 \begin{equation}
  r_{\A,\ell}(n) \leq \sum_{\substack{k \in S \\ S\in \mathcal{M}}} r_{\A,\ell-1}(n-k) \leq \ell! \, \widehat{r}_{\A,\ell}(n)\, \Big(\max_{k\leq n} r_{\A,\ell-1}(k) \Big). \label{mchdeg}
 \end{equation}
 Whenever a representation function refers to the size of a maximum disjoint family, we append a `` $\widehat{\ }$ ’’ to its notation; for instance, $\widehat{\rho}^{(\delta\text{-small})}_{\A,\ell}$ denotes the size of the largest maxdisfam of $\delta$-small exact representations.

 \begin{lem}\label{prodHat}
  For every $2\leq \ell\leq h$, we have
  \[ \mathrm{R}_{\A,\ell}(n) \ll \widehat{\mathrm{R}}_{\A,\ell}(n)\,
    \Big(\max_{k\leq n} \widehat{r}_{\A,\ell-1}(k) \Big)
    \cdots
    \Big(\max_{k\leq n} \widehat{r}_{\A,2}(k) \Big), \]
  where $\mathrm{R}$ stands for $r$, $\rho$, or $\rho^{(\delta\textnormal{-small})}$.
 \end{lem}
 \begin{proof}
  Applying \eqref{mchdeg} to $\mathrm{R}$, the same recursive bound holds for $r_{\A,\ell-t}$ with $1\leq t\leq \ell-2$. Thus
  \[ \mathrm{R}_{\A,\ell}(n)
    \ll \widehat{\mathrm{R}}_{\A,\ell}(n)
    \Big(\max_{k\leq n} \widehat{r}_{\A,\ell-1}(k) \Big)
    \cdots
    \Big(\max_{k\leq n} \widehat{r}_{\A,3}(k) \Big)
    \Big(\max_{k\leq n} r_{\A,2}(k) \Big). \]
  Since $r_{\A,2}(k) \leq 2\,\widehat{r}_{\A,2}(k)$, the claim follows.
 \end{proof}

 Lemma \ref{prodHat} will be used in conjunction with the following standard result \cite[Lemma~1]{erdtet90}.

 \begin{lem}[Disjointness lemma]\label{disjlm}
  Let $\mathscr{E} = \{E_1,E_2,\ldots\}$ be a family of events, and define $S := \sum_{E\in\mathscr{E}} \mathbbm{1}_E$. If $\E(S) < \infty$, then for every $k\geq 1$ we have
  \[ \Pr(\exists\mathcal{D}\subseteq \mathscr{E} \textnormal{ disfam with }|\mathcal{D}| = k)
    \leq \sum_{\substack{\mathcal{J}\subseteq \mathscr{E}\textnormal{ disfam} \\ |\mathcal{J}| = k}}
    \Pr\bigg( \bigwedge_{E\in\mathcal{J}} E\bigg)
    \leq \frac{\E(S)^{k}}{k!}. \]
 \end{lem}
 \begin{proof}
  \begin{align*}
   \sum_{\substack{\mathcal{J}\subseteq \mathscr{E}\textnormal{ disfam} \\ |\mathcal{J}| = k}}
   \Pr\bigg( \bigwedge_{E\in\mathcal{J}} E\bigg)
   &= \sum_{\substack{\mathcal{J}\subseteq \mathscr{E}\textnormal{ disfam} \\ |\mathcal{J}| = k}}
     \prod_{E\in\mathcal{J}} \Pr(E)
   \leq \frac{1}{k!}\bigg( \sum_{E\in \mathscr{E}} \Pr(E) \bigg)^{k}
   = \frac{\E(S)^{k}}{k!}. \qedhere
  \end{align*}
 \end{proof}

 \begin{xrem}
  If a maxdisfam of size greater than $k$ exists, then a disjoint family of size $k$ exists as well. Hence Lemma \ref{disjlm} gives, using $k!\geq k^k e^{-k}$,
  \begin{align*}
   \Pr(\exists \text{maxdisfam of size} \geq x)
   &\leq \bigg(\frac{e\,\E}{\lceil x\rceil}\bigg)^{\lceil x\rceil} \qquad (\text{for real }x\geq 1) \\
   &\leq \bigg(\frac{e\,\E}{x}\bigg)^{x} \phantom{\bigg(\frac{e\,\E}{\lceil x\rceil}\bigg)^{\lceil x\rceil}} \negphantom{\bigg(\frac{e\,\E}{x}\bigg)^{x}} \qquad (\text{for real }x\geq 1+\E).
  \end{align*}
 \end{xrem}

%%%%%%%%%%%%%%%%%%%%%%%%%%%%%%%%%%%%%
\subsection{Theorem \ref{MT2}: Case \texorpdfstring{$\kappa = 0$}{kappa =0}} 
 In this case $\omega=1/h$ and $F(x)=x^{o(1)}$. Fix parameters $0<\delta<\eta<\tfrac{1}{h(h-1)}$. We treat three contributions separately: the $\delta$-normal term $\rho_{\A,h}^{(\delta\text{-normal})}(n)$, the $\delta$-small term $\rho_{\A,h}^{(\delta\text{-small})}(n)$, and the non-exact terms $\rho_{\A,\ell}^{(c_1,\dots,c_\ell)}(n)$ with $1\leq\ell\leq h-1$. By \eqref{decomprho}, showing that the first term is a.s. asymptotic to its expectation and that the remaining terms are a.s. $O(1)$ (indeed $o(F(n))$ since $F(n)\to\infty$) will prove the theorem.

\medskip\noindent
 $\textbf{(1)}~\text{\underline{Concentration for the $\delta$-normal part.}}$
 The function $\rho_{\A,h}^{(\delta\text{-normal})}(n)$ is a homogeneous, simple boolean polynomial of degree $h$. Since each monomial of $\rho_{\A,h}(n)$ appears at most $h!$ times, we verify the hypotheses of Theorem \ref{vuORI} (Vu) for the normalized polynomial $\tfrac{1}{h!}\rho_{\A,h}^{(\delta\text{-normal})}(n)$.

 Let $1\leq j\leq h-1$. By Lemma \ref{psav} we have for $n^{\delta}\leq k\leq n$
 \[ \E(r_{\A,h-j}(k))\ll k^{-\eta}F(k). \]
 Hence
 \[ \E_j(\rho_{\A,h}^{(\delta\text{-normal})}(n)) \leq \max_{n^{\delta}\leq k\leq n} \E(r_{\A,h-j}(k))
 \ll n^{-\eta\delta+o(1)}. \]
 Thus, choosing $0<\alpha<\eta\delta$ we have the smallness condition $\E_1,\dots,\E_{h-1}\leq n^{-\alpha}$ required by Theorem \ref{vuORI}. Take $\gamma=2$ and let $\Lambda=\Lambda(h,\alpha,\gamma)$ be the constant provided by Vu. Put
 \[ \lambda := 32h\Lambda\log n. \]
 Then Theorem \ref{vuORI} yields
 \[ \Pr\Big(|\rho_{\A,h}^{(\delta\text{-normal})}(n)-\E(\rho_{\A,h}^{(\delta\text{-normal})}(n))| \geq (h!\lambda\,\E(\rho_{\A,h}^{(\delta\text{-normal})}(n)))^{1/2}\Big)
 \ll n^{-2}. \]

 If $F(n)/\log n\to\infty$, then by Lemmas \ref{mainlm} and \ref{expdsml} the expectation $\E(\rho_{\A,h}^{(\delta\text{-normal})}(n))\asymp F(n)$, so $(\lambda \E)^{1/2}=o(\E)$ and the deviation bound is negligible; hence by Borel--Cantelli
 \[ \rho_{\A,h}^{(\delta\text{-normal})}(n)\stackrel{\text{a.s.}}{\sim} \E(\rho_{\A,h}^{(\delta\text{-normal})}(n))\sim c^h\mathfrak{S}(n)F(n), \]
 and we choose $c=1$. If instead only $F(n)\gg \log n$ (but $F(n)/\log n$ need not tend to $\infty$), one may choose the normalization constant $c>0$ sufficiently large when defining the model so that $\E(\rho_{\A,h}(n)) > \lambda$ for large $n$; the same argument then yields $\rho_{\A,h}^{(\delta\text{-normal})}(n)\stackrel{\text{a.s.}}{\asymp}F(n)$.

 \medskip\noindent
 $\textbf{(2)}~\text{\underline{The $\delta$-small contribution.}}$
  We show the $\delta$-small contribution is a.s.\ bounded. First note the relevant expectation bounds (from Lemmas \ref{psav} and \ref{expdsml}):
 \[ \E(r_{\A,\ell}(k))\ll k^{-\eta+o(1)} \quad(1\leq\ell\leq h-1),
  \qquad
  \E(\rho_{\A,h}^{(\delta\text{-small})}(n))\ll n^{\delta/h-\eta+o(1)}. \]
 Fix an integer \(T\ge1\). For any fixed \(k\leq n\), Lemma \ref{disjlm} gives
 \[ \Pr(\widehat{r}_{\A,\ell}(k)\geq T) \leq \bigg(\frac{e\,\E(r_{\A,\ell}(k))}{T}\bigg)^{T}
  \ll \bigg(\frac{e}{T}\bigg)^{T} k^{-T\eta+o(1)}. \]
 Summing this bound over \(k\leq n\) yields
 \[ \Pr\Big(\max_{k\leq n}\widehat{r}_{\A,\ell}(k)\geq T\Big) \leq \sum_{k\leq n} \Pr(\widehat{r}_{\A,\ell}(k)\geq T) \ll \bigg(\frac{e}{T}\bigg)^{T} n^{1-T\eta+o(1)}. \]
 The same argument applied to the $\delta$-small family gives
 \[ \Pr(\widehat{\rho}_{\A,h}^{(\delta\text{-small})}(n)\geq T) \ll \bigg(\frac{e}{T}\bigg)^{T} n^{-T(\eta-\delta/h)+o(1)}. \]
 Since \(\eta>\delta/h\), both exponents \(1-T\eta\) and \(-T(\eta-\delta/h)\) are negative for large $T$. Taking $T> 3/(\eta-\delta/h)$, both right-hand sides are \(\ll n^{-2+o(1)}\), hence summable in \(n\). By Borel--Cantelli it follows that for every fixed \(1 \leq \ell\leq h-1\),
 \[ \max_{k\leq n}\widehat{r}_{\A,\ell}(k)\stackrel{\text{a.s.}}{\ll}1, \qquad \widehat{\rho}_{\A,h}^{(\delta\text{-small})}(n)\stackrel{\text{a.s.}}{\ll}1. \]
 Substituting these bounds into Lemma \ref{prodHat} yields
 \[ \rho_{\A,h}^{(\delta\text{-small})}(n) \ll \widehat{\rho}_{\A,h}^{(\delta\text{-small})}(n)
     \Big(\max_{k\leq n}\widehat{r}_{\A,h-1}(k)\Big)\cdots \Big(\max_{k\leq n}\widehat{r}_{\A,2}(k)\Big)\stackrel{\text{a.s.}}{\ll}1, \]
 as required.

 \medskip\noindent
 $\textbf{(3)}~\text{\underline{Remaining non-exact terms.}}$
 Finally, we treat each fixed non-exact tuple $(c_1,\dots,c_\ell)$ with $1\leq \ell\leq h-1$. If $h=2$ the claim is trivial. For $h\geq 3$ the same disjointness argument applies to show that the maximal disjoint family for these non-exact representations is a.s. bounded, since their expectations are $\ll n^{-\eta+o(1)}$ by Lemma \ref{psav}. Using the (same) product bound as in Lemma \ref{prodHat} (or its obvious variant adapted to these representations) we obtain
 \[ \rho_{\A,\ell}^{(c_1,\dots,c_\ell)}(n)\stackrel{\text{a.s.}}{\ll}1. \]

 Combining the three parts, \(\rho_{\A,h}^{(\delta\text{-normal})}(n)\) provides the main contribution to $r_{\A,h}(n)$ and all other contributions are a.s. negligible. This completes the proof.\hfill$\square$

%%%%%%%%%%%%%%%%%%%%%%%%%%%%%%%%%%%%%
\subsection{Proof of Corollary \ref{cor-tech}}  
 The condition $\mathcal{L}(x^{\lambda})\asymp_{\lambda}\mathcal{L}(x)$ implies $\mathcal{L}(x)\asymp \psi(x)$ for some slowly varying function $\psi$, by \cite[Theorems 1.3.1 and 2.2.7]{bingham89}. Indeed, this assumption means that $\mathcal{L}(e^{x}) \asymp_{\lambda} \mathcal{L}(e^{\lambda x})$, i.e.\ that the function $\mathcal{L}(e^{x})$ is $O$-regularly varying. By the representation theorem for $O$-regularly varying functions \cite[Theorem 2.2.7]{bingham89}, there exist bounded measurable functions $\alpha,\xi:\R_{>0}\to\R$ such that
 \[  \psi(e^{x})=\exp\bigg(\alpha(x)+\int_{1}^{x}\frac{\xi(t)}{t}\,\mathrm{d}t\bigg). \]
 Writing $y=e^{x}$ and changing variables gives
 \[ \mathcal{L}(y)\asymp \exp\bigg(\int_{1}^{\log y}\frac{\xi(t)}{t}\,\mathrm{d}t\bigg) = \exp\bigg(\int_{1}^{y}\frac{\xi(\log u)}{u\log u}\,\mathrm{d}u\bigg). \]
 Setting $\epsilon(u):=\xi(\log u)/\log u$, we have $\epsilon(u)\to0$ as $u\to\infty$ since $\xi$ is bounded. Therefore,
 \[ \mathcal{L}(x)\asymp \psi(x):=\exp\!\bigg(\int_{1}^{x}\frac{\epsilon(t)}{t}\,\mathrm{d}t\bigg), \]
 and by \cite[Theorem 1.3.1]{bingham89}, this $\psi$ is slowly varying.

 Define now $G(x):=x^{\kappa}\psi(x)$ and $g(x):=(xG(x))^{1/h}$. Let $0<\eta<\tfrac{1}{h(h-1)}$. Applying Corollary \ref{cor1} to the function $g$ yields
 \[ \sum_{\substack{x_1,\ldots,x_h\in B \\ x_1+\cdots+x_h=n}}
    \frac{g(x_1)}{x_1^{\beta}\vartheta(x_1)}\cdots
    \frac{g(x_h)}{x_h^{\beta}\vartheta(x_h)}
  = \sum_{\substack{x_1,\ldots,x_h\in B \\ x_1+\cdots+x_h=n \\ \forall j,\,x_j\ge n^{\eta/2\omega}}}
    \frac{g(x_1)}{x_1^{\beta}\vartheta(x_1)}\cdots
    \frac{g(x_h)}{x_h^{\beta}\vartheta(x_h)}
  + O(n^{-\frac{\eta}{2}+o(1)}G(n)). \]
 By uniform convergence for $O$-regularly varying functions \cite[Theorem 2.0.8]{bingham89} (applied to $x\mapsto\psi(e^{x})$), we have $\psi(x_j)\asymp\psi(n)$ uniformly for $n^{\eta/2\omega}\leq x_j\leq n$. Thus
 \[ \sum_{\substack{x_1,\ldots,x_h\in B \\ x_1+\cdots+x_h=n \\ \forall j,\,x_j\geq n^{\eta/2\omega}}}
    \frac{g(x_1)}{x_1^{\beta}\vartheta(x_1)}\cdots
    \frac{g(x_h)}{x_h^{\beta}\vartheta(x_h)}
    \asymp \psi(n) \sum_{\substack{x_1,\ldots,x_h\in B \\ x_1+\cdots+x_h=n \\ \forall j,\,x_j\geq n^{\eta/2\omega}}} \frac{x_1^{\frac{1+\kappa}{h}}}{x_1^{\beta}\vartheta(x_1)}\cdots \frac{x_h^{\frac{1+\kappa}{h}}}{x_h^{\beta}\vartheta(x_h)}. \]
 Applying Corollary \ref{cor1} again with $f(x)=x^{(1+\kappa)/h}$ yields
 \[ \sum_{\substack{x_1,\ldots,x_h\in B \\ x_1+\cdots+x_h=n \\ \forall j,\,x_j\ge n^{\eta/2\omega}}}
    \frac{x_1^{\frac{1+\kappa}{h}}}{x_1^{\beta}\vartheta(x_1)}\cdots
    \frac{x_h^{\frac{1+\kappa}{h}}}{x_h^{\beta}\vartheta(x_h)}
  = (1+o(1))\,\mathfrak{S}(n)\,n^{\kappa}
    + O(n^{-\frac{\eta}{2}+o(1)}n^{\kappa}), \]
 by \eqref{mainest} with $F(x)=x^{\kappa}$, for $n\in\mathscr{S}$. Consequently,
 \[ \sum_{\substack{x_1,\ldots,x_h\in B \\ x_1+\cdots+x_h=n}} \frac{g(x_1)}{x_1^{\beta}\vartheta(x_1)}\cdots
    \frac{g(x_h)}{x_h^{\beta}\vartheta(x_h)} = \mathfrak{T}(n)\,G(n), \]
 where $\mathfrak{T}(n)\asymp (1+o(1))\,\mathfrak{S}(n) + O(n^{-\frac{\eta}{2}+o(1)})\asymp 1$ on $\mathscr{S}$. Since $G(n)\gg\log n$ and $G(n)\ll n^{h\beta-1}\vartheta(n)^h$ by assumption, Theorem \ref{MT2} (ii) applies to $G$, yielding a subset $A\subseteq B$ with
 \[ r_{A,h}(n)\asymp G(n)\asymp n^{\kappa}\mathcal{L}(n) \qquad(n\in\mathscr{S}), \]
 completing the proof.\hfill$\square$

%%%%%%%%%%%%%%%%%%%%%%%%%%%%%%%%%%%%%
\subsection{Proof of Corollary \ref{cor-tech2}}  
 Write $g(x) = x^{\omega}\phi(x)$, so $\phi(x) := \psi(x)^{1/h}$. By the uniform convergence theorem (BGT \cite[Theorem 1.2.1]{bingham89}), given $0< \mu < 1$ and $\eps > 0$ there exists $x_{\mu,\eps}$ such that, for every $x\geq x_{\mu,\eps}$,
 \[ \bigg|\frac{\phi(\lambda x)}{\phi(x)} - 1\bigg| < \eps, \qquad \forall \lambda\in[\mu,1]. \]
 Taking $\mu=1/j$ and $\eps=1/j$ with $j\in\Z_{\geq 1}$, we define a non-decreasing function $\xi(x)\to\infty$ by setting $\xi(x):=j$ whenever $x\in [x_{\frac{1}{j},\frac{1}{j}},\, x_{\frac{1}{j+1},\frac{1}{j+1}})$. Then
 \begin{equation}\label{eps0}
  \frac{\phi(y)}{\phi(x)}\to 1 \text{ uniformly for } y\in\bigg[\frac{x}{\xi(x)},x\bigg].
 \end{equation}
 
 We split the weighted representation sum as $S_1+S_2$, where
 \[ S_1 := \sum_{\substack{x_1,\ldots,x_h\in B \\ x_1+\cdots+x_h = n \\ \exists j \,|\, x_{j}< n/\xi(n)}} \frac{g(x_1)}{x_1^{\beta}\vartheta(x_1)}\cdots\frac{g(x_h)}{x_h^{\beta}\vartheta(x_h)}, \qquad
    S_2 := \sum_{\substack{x_1,\ldots,x_h\in B \\ x_1+\cdots+x_h = n \\ \forall j,\, x_{j}\geq n/\xi(n)}} \frac{g(x_1)}{x_1^{\beta}\vartheta(x_1)}\cdots\frac{g(x_h)}{x_h^{\beta}\vartheta(x_h)}. \]
 
 \medskip\noindent
 $\textbf{(1)}~\text{\underline{The range with a small variable.}}$
 We show that $S_1=o(G(n))$. By symmetry, it suffices to bound the contribution of solutions with $x_1< n/\xi(n)$. By Corollary \ref{cor1}, for any fixed $X\geq 1$ the contribution to $S_1$ from tuples with $\min_j x_j<X$ is $O(n^{-\eta}G(n))$ for some $\eta>0$. Hence, choosing $X = X_{\delta}$ as in the Potter bound below, we may restrict to tuples with $x_j\geq X_{\delta}$ for all $j$ at the cost of an error $o(G(n))$.
 
 Fix the $\delta > 0$ appearing in the hypothesis. By Potter bounds (BGT \cite[Theorem 1.5.6 (i)]{bingham89}), there exists $X_{\delta}\geq 1$ such that $\phi(y) \leq 2 (x/y)^{\delta} \phi(x)$ for $x\geq y\geq X_{\delta}$. Hence
 \begin{align*}
  \sum_{\substack{x_1,\ldots,x_h\in B \\ x_1+\cdots+x_h = n \\ x_{1}< n/\xi(n)}} \, &\frac{g(x_1)}{x_1^{\beta}\vartheta(x_1)}\cdots \frac{g(x_h)}{x_h^{\beta}\vartheta(x_h)} \\
  &\ll \sum_{\substack{x_1,\ldots,x_h\in B \\ x_1+\cdots+x_h = n \\ x_{1}< n/\xi(n)}} \frac{x_1^{\omega}}{x_1^{\beta}\vartheta(x_1)}\cdots\frac{x_h^{\omega}}{x_h^{\beta}\vartheta(x_h)} \bigg(\frac{n/\xi(n)}{x_1} \frac{n}{x_2}\cdots \frac{n}{x_h} \bigg)^{\delta} \phi(n/\xi(n))\, \phi(n)^{h-1} \\
  &\ll \frac{n^{h\delta} \phi(n)^h }{\xi(n)^{\delta}} \sum_{\substack{x_1,\ldots,x_h\in B \\ x_1+\cdots+x_h = n}} \frac{x_1^{\omega - \delta}}{x_1^{\beta}\vartheta(x_1)}\cdots\frac{x_h^{\omega - \delta}}{x_h^{\beta}\vartheta(x_h)},
 \end{align*}
 since $\phi(n/\xi(n)) \sim \phi(n)$. By the hypothesis, the sum above is bounded by $O(n^{h(\omega-\delta)-1})$, thus
 \[ \sum_{\substack{x_1,\ldots,x_h\in B \\ x_1+\cdots+x_h = n \\ x_{1}< n/\xi(n)}} \frac{g(x_1)}{x_1^{\beta}\vartheta(x_1)}\cdots \frac{g(x_h)}{x_h^{\beta}\vartheta(x_h)} \ll \frac{n^{h\omega-1} \phi(n)^h}{\xi(n)^{\delta}}  = \frac{G(n)}{\xi(n)^{\delta}}. \]
 Since $\delta > 0$ and $\xi(n) \to \infty$, it follows that the sum, and hence $S_1$, is $o(G(n))$.

 %%%%%%%%%%%%%%%%
 \medskip\noindent
 $\textbf{(2)}~\text{\underline{The main range.}}$
 By \eqref{eps0},
 \begin{align}
  S_2 &= \phi(n)^h \sum_{\substack{x_1,\ldots,x_h\in B \\ x_1+\cdots+x_h = n \\ \forall j,\, x_{j}\geq n/\xi(n)}} \frac{x_1^{\omega}}{x_1^{\beta}\vartheta(x_1)}\cdots\frac{x_h^{\omega}}{x_h^{\beta}\vartheta(x_h)} \bigg(\frac{\phi(x_1)}{\phi(n)}\cdots \frac{\phi(x_h)}{\phi(n)}\bigg) \nonumber \\
  &\sim \phi(n)^h \sum_{\substack{x_1,\ldots,x_h\in B \\ x_1+\cdots+x_h = n \\ \forall j,\, x_{j}\geq n/\xi(n)}} \frac{x_1^{\omega}}{x_1^{\beta}\vartheta(x_1)}\cdots\frac{x_h^{\omega}}{x_h^{\beta}\vartheta(x_h)}. \label{eqrwst}
 \end{align}
 Moreover, the contribution to the power-weighted sum from tuples with $\min_j x_j< n/\xi(n)$ is $o(n^{h\omega-1})$ by the argument of \textbf{(1)} applied with $\phi\equiv 1$. Hence the restriction $x_j\geq n/\xi(n)$ may be removed from \eqref{eqrwst} at the cost of an error $o(n^{h\omega - 1} \phi(n)^h) = o(G(n))$. Since we suppose \eqref{mainest} holds with $F(x) = x^{\kappa}$, it follows from \eqref{eqrwst} that
 \[ S_2 \sim \mathfrak{S}(n)\,G(n) + o(G(n)). \]
 But $\mathfrak{S}(n)\asymp 1$ and thus $o(G(n))=o(\mathfrak{S}(n)\,G(n))$ for $n\in\mathscr{S}$. Together with $S_1=o(G(n))$, this proves that $S_1+S_2 \sim \mathfrak{S}(n)\,G(n)$ for $n\in\mathscr{S}$. Parts (i) and (ii) then follow immediately from Theorem \ref{MT2}. \hfill$\square$

%%%%%%%%%%%%%%%%%%%%%%%%%%%%%%%%%%%%%%%%%%%%%%%%%%%%
\section{Waring subbases: Theorem \ref{MT1}}\label{waring1}
 In this section we prove Theorem \ref{MT1}. The central analytic input is a weighted asymptotic formula for Waring's problem (Theorem \ref{wolem}), a refinement of \cite{woo04} in which the admissible range for the number of variables is reduced to $h\geq k^2-k+O(\sqrt{k})$. The circle method argument is the same as in \cite{woo04}; our only modification is that this admissible range of $h$ is sharpened by appealing to the mean value estimate in Lemma \ref{war-hua} in the minor arc analysis. The same lemma also supplies the Hua-type hypothesis \eqref{hua-type} for $B=\N^k$ needed in the application of the general subbasis principle.
 
 \begin{lem}[Mean value estimate]\label{war-hua}
  Let $k\geq 1$, and define
  \begin{equation}\label{h0k}
   H_0 = H_0(k) := \begin{cases}
                    2^{k-1}, & 1\leq k\leq 4; \\
                    \frac{1}{2}k(k-1) + \lfloor \sqrt{2k+2}\rfloor, & k\geq 5.
                   \end{cases}
  \end{equation}
  Then, for every integer $\ell \geq H_0$,
  \[ \int_{0}^{1} \bigg| \sum_{n\leq x} e(\alpha n^{k})\bigg|^{2\ell} \,\mathrm{d}\alpha \ll x^{2\ell - k}(\log x)^M \]
  for some constant $M = M(k)>0$. When $k\geq 5$, the factor $(\log x)^M$ can be dropped.
 \end{lem}
 \begin{proof}
  For $1\leq k\leq 4$, this is Hua's inequality (cf. \cite[Theorem 4]{hua65}). Assume now that $k\geq 5$. By \cite[Corollary 14.8 and Eq. (14.26)]{woo19}, since $k-\lfloor \sqrt{2k+2}\rfloor \geq 2$ for $k\geq 5$, one has
  \begin{equation}\label{eaxkax}
   \int_{0}^{1} \int_{0}^{1} \bigg|\sum_{n\leq x} e(\alpha_1 n^k + \alpha_2 n) \bigg|^s \,\mathrm{d}\alpha_1 \,\mathrm{d}\alpha_2 \ll x^{s-k-1}
  \end{equation}
  for every real $s \geq k(k-1) + 2\lfloor \sqrt{2k+2}\rfloor - 1 = 2H_0-1$. In particular, we may take $s=2\ell$ for any integer $\ell\geq H_0$. For such $\ell$, we have
  \[ \int_0^1 \bigg|\sum_{n\leq x} e(\alpha n^k)\bigg|^{2\ell}\,\mathrm{d}\alpha = \sum_{|m|\leq \ell x}\int_0^1\int_0^1 \bigg|\sum_{n\leq x} e(\alpha_1 n^k+\alpha_2 n)\bigg|^{2\ell} e(-m\alpha_2)\,\mathrm{d}\alpha_1\,\mathrm{d}\alpha_2. \]
  Indeed, expanding and applying orthogonality in $\alpha_1,\alpha_2$ shows that, for each $m\in\Z$,
  \[ \int_0^1\int_0^1 \bigg|\sum_{n\leq x} e(\alpha_1 n^k+\alpha_2 n)\bigg|^{2\ell} e(-m\alpha_2)\,\mathrm{d}\alpha_1\,\mathrm{d}\alpha_2
  = \sum_{\substack{n_1,\ldots,n_\ell\leq x \\ r_1,\ldots,r_\ell\leq x}} \mathbbm{1}_{\sum n_i^k=\sum r_i^k}\,\mathbbm{1}_{\sum n_i-\sum r_i=m}, \]
  and since $\sum n_i-\sum r_i\in[-\ell x,\ell x]$, summing over $|m|\leq \ell x$ yields the claimed identity. Taking absolute values gives
  \[ \int_0^1 \bigg|\sum_{n\leq x} e(\alpha n^k)\bigg|^{2\ell}\,\mathrm{d}\alpha \leq (2\ell x+1)\int_0^1\int_0^1 \bigg|\sum_{n\leq x} e(\alpha_1 n^k+\alpha_2 n)\bigg|^{2\ell}\,\mathrm{d}\alpha_1\,\mathrm{d}\alpha_2, \]
  and the claim follows from \eqref{eaxkax} with $s=2\ell$.
 \end{proof}

 \begin{xrem}
  From Lemma \ref{war-hua}, by Parseval's identity, we have
  \[ \sum_{n\leq x} r_{\N^{k},\ell}(n)^{2} \leq \int_{0}^{1} \bigg| \sum_{n\leq x^{1/k}} e(\alpha n^{k})\bigg|^{2\ell} \,\mathrm{d}\alpha \ll x^{2\ell/k - 1 + o(1)}, \]
  which is in the shape of \eqref{hua-type}.
 \end{xrem}
 
 It remains to establish the weighted asymptotic \eqref{mainest} for $\N^k$ in the power case $F(x)=x^\kappa$. For this we prove the following theorem. 
 
 \begin{thm}\label{wolem}
  Let $H_0 = H_0(k)$ be as in \eqref{h0k}, and let $h\geq 2H_{0} + 1$. Let $\delta>0$ be any real number with 
  \[ \delta < \frac{h-2H_0}{2 h(h-1) H_0}. \]
  Then, for any $\omega \geq 1/h - \delta$,
  \[ \sum_{\substack{x_1,\ldots,x_h\in \N^k \\ x_1+\cdots+x_h = N}} (x_1\cdots x_h)^{\omega - \frac{1}{k}} = \mathfrak{S}_{k,h}(N)\,\frac{1}{k^h}\frac{\Gamma(\omega)^h}{\Gamma(h\omega)} \, N^{h\omega - 1} + O(N^{h\omega -1 -\tau}) \]
  for some $\tau = \tau(k,h,\delta) >0$,\footnote{An explicit value of $\tau$ can be obtained from Propositions \ref{majorwar} and \ref{minorwar}, although we did not try to optimize it.} where the singular series $\mathfrak{S}_{k,h}(N)$ is defined as in \eqref{singser}.
 \end{thm}
 
 Since $|\N^k\cap [0,x]| = x^{1/k} + O(1)$, clearly $\N^k$ satisfies \eqref{oregvar}, with $\beta=1/k$. With Lemma \ref{war-hua}, $\N^k$ satisfies \eqref{hua-type}. Thus, assuming Theorem \ref{wolem}, Theorem \ref{MT1} follows from Corollary \ref{cor-tech2} applied to every $0\leq \kappa \leq h/k-1$.

%%%%%%%%% 
\subsection{Setup for Theorem \ref{wolem}}
 We now prove Theorem \ref{wolem}. Fix $k\geq 2$, let $h\geq 2H_0 + 1$ be an integer, let $\omega\geq 1/h - \delta$ be a real number, and let $N$ be a large integer. Define%\footnote{In the weighted sums below, the element $0$ is omitted; this has no effect on the representation estimates.}
 \begin{equation}\label{Twaring}
  T(\alpha;x) = \sum_{n\leq x} n^{\omega}\, \frac{\mathbbm{1}_{\N^k}(n)}{n^{1/k}} \, e(\alpha n), \qquad T^{\sharp}(\alpha;x) = \sum_{x/h\leq n\leq x} n^{\omega}\, \frac{\mathbbm{1}_{\N^k}(n)}{n^{1/k}} \, e(\alpha n),
 \end{equation}
 so that, by orthogonality,
 \[ \sum_{\substack{x_1,\ldots,x_h \in \N^k \\ x_1+\cdots+x_h = N}} (x_1\cdots x_h)^{\omega-\frac{1}{k}} = \int_{0}^{1} T(\alpha;N)^{h}\, e(-N\alpha)\,\mathrm{d}\alpha. \]
 If $x_1+\cdots+x_h = N$, then $x_i\geq N/h$ for at least one $x_i$. Hence,
 \[ \int_{0}^{1} (T(\alpha;N)-T^{\sharp}(\alpha;N))^h\, e(-N\alpha)\,\mathrm{d}\alpha = 0. \]
 Therefore, 
 \begin{align}
  \sum_{\substack{x_1,\ldots,x_h \in \N^k \\ x_1+\cdots+x_h = N}} (x_1\cdots x_h)^{\omega-\frac{1}{k}}
  &= \int_{0}^{1} \big(T(\alpha;N)^h - (T(\alpha;N)-T^{\sharp}(\alpha;N))^h\big)\, e(-N\alpha)\,\mathrm{d}\alpha \nonumber \\
  &= \sum_{j=1}^{h} (-1)^{j+1}\binom{h}{j} \int_{0}^{1} T^{\sharp}(\alpha;N)^{j}\,T(\alpha;N)^{h-j}\, e(-N\alpha)\,\mathrm{d}\alpha. \label{sothat1}
 \end{align}
 
  As is standard in the circle method, we divide $[0,1]$ into major arcs $\mathfrak{M}$ and minor arcs $\mathfrak{m}$, where the major arcs $\mathfrak{M}=\mathfrak{M}_N$ are defined as
 \[ \mathfrak{M} := \bigsqcup_{q\leq L}\bigsqcup_{\substack{a=1 \\ (a,q)=1}}^{q} \mathfrak{M}(q,a), \quad \mathfrak{M}(q,a) := \bigg\{\alpha \in [0,1] ~\bigg|~ \bigg\|\alpha -\frac{a}{q}\bigg\| \leq \frac{L}{N} \bigg\}  \]
 where $\|\cdot\|$ denotes distance to the nearest integer, and for a parameter
 \begin{equation}
  L = L(N) := N^{\nu}, \qquad \nu = \nu(\omega) := \min\Big\{\frac{1}{7k},\, \frac{\omega}{7}\Big\}. \label{nuchoice}
 \end{equation}
 The minor arcs are defined as the complement, $\mathfrak{m} = \mathfrak{m}_N := [0,1]\setminus \mathfrak{M}$. Splitting the integral in \eqref{sothat1} as $\int_{0}^{1} = \int_{\mathfrak{M}} + \int_{\mathfrak{m}}$, we anticipate that the contribution from the minor arcs is negligible. By \eqref{sothat1}, Theorem \ref{wolem} will follow from a major arc evaluation (Proposition \ref{majorwar}) and a minor arc bound saving a power of $N$ (Proposition \ref{minorwar}).

%%%%%%%%%%%%
\subsection{Major arcs} 
 The goal of this subsection is to prove the following:
 
 \begin{prop}[Major arcs]\label{majorwar}
  We have
  \begin{align*}
   \sum_{j=1}^{h} (-1)^{j+1}\binom{h}{j} \int_{\mathfrak{M}} T^{\sharp}(\alpha;N)^{j}\, T(\alpha;\,N)^{h-j}\,&e(-N\alpha)\,\mathrm{d}\alpha \\
   &= \mathfrak{S}_{k,h}(N)\,\frac{1}{k^h}\frac{\Gamma(\omega)^h}{\Gamma(h\omega)} \, N^{h\omega-1} + O(N^{h\omega-1-\eta}),
  \end{align*}
  where $\mathfrak{S}_{k,h}(N)$ is as in \eqref{singser}, and $\eta := \nu \min\{1/k, (h-1)\omega\}$.
 \end{prop}

 In order to calculate the integral in Proposition \ref{majorwar}, we will approximate $T$ and $T^{\sharp}$ in terms of, respectively,
 \begin{equation}
  U(\vartheta; x) := \sum_{n\leq x} n^{\omega -1} e(n\vartheta), \qquad U^{\sharp}(\vartheta; x) := \sum_{x/h\leq n\leq x} n^{\omega -1} e(n\vartheta). \label{defU}
 \end{equation}
 
 \begin{lem}\label{war-est2}
  Let $q\leq L$, and $1\leq a\leq q$ with $(a,q)=1$. For $\alpha \in \mathfrak{M}(q,a)$, write $\vartheta := \alpha - \frac{a}{q}$. Then
  \[ T(\alpha; N) = \frac{S(q,a)}{kq} U(\vartheta; N) + O(N^{\omega-4\nu}),\quad T^{\sharp}(\alpha; N) = \frac{S(q,a)}{kq} U^{\sharp}(\vartheta; N) + O(N^{\omega- 4\nu}) \]
  uniformly for $q\leq L$, where $S(q,a) = \sum_{r=1}^{q} e(ar^{k}/q)$.
 \end{lem}
 \begin{proof}
  We prove the lemma for $T$; the argument for $T^{\sharp}$ is analogous.
  
  Let $P_k(q,r) := \#\{1\leq s\leq q ~|~ s^k \equiv r\pmod{q}\}$. For every $1\leq r\leq q$, we have
  \begin{align*}
   \sum_{\substack{n\leq x\\ n\equiv r(q)}} n^{\omega}\,\frac{\mathbbm{1}_{\N^k}(n)}{n^{1/k}} &= \sum_{\substack{1\leq s\leq q \\ s^k \equiv r (q)}} \sum_{m\leq \frac{x^{1/k} - s}{q}} (qm+s)^{k\omega-1} \\
   &= q^{k\omega-1} \sum_{\substack{1\leq s\leq q \\ s^k \equiv r (q)}} \sum_{m\leq \frac{x^{1/k} - s}{q}} \bigg(m+\frac{s}{q}\bigg)^{k\omega-1} \\
   &= q^{k\omega-1} P_k(q,r) \bigg(\frac{(x^{1/k}/q)^{k\omega}}{k\omega} + O\big(1 + (x^{1/k}/q)^{k\omega-1}\big) \bigg) \\
   &= \frac{P_k(q,r)}{kq} \frac{x^{\omega}}{\omega} + O(q^{k\omega} + q x^{\omega- \frac{1}{k}})
  \end{align*}
  Hence,
  \begin{align}
   T\bigg(\frac{a}{q}; x\bigg) = \sum_{n\leq x} n^{\omega}\,\frac{\mathbbm{1}_{\N^k}(n)}{n^{1/k}}\, e\bigg(\frac{na}{q}\bigg) &= \sum_{r=1}^{q} \Bigg(\sum_{\substack{n\leq x\\ n\equiv r(q)}} n^{\omega}\,\frac{\mathbbm{1}_{\N^k}(n)}{n^{1/k}}\Bigg) e\bigg(\frac{ra}{q}\bigg) \nonumber \\
   &= \Bigg(\frac{1}{kq}\sum_{r=1}^{q} P_k(q,r)\,e\bigg(\frac{ra}{q} \bigg) \Bigg) \frac{x^{\omega}}{\omega} + O(q^{k\omega+1} + q^2x^{\omega- \frac{1}{k}}) \nonumber \\
   &= \frac{S(q,a)}{kq} \frac{x^{\omega}}{\omega} + O(q^{k\omega+1} + q^2 x^{\omega- \frac{1}{k}}). \label{gxab1}
  \end{align}  
  (Note that $T^{\sharp}(\frac{a}{q};x) = T(\frac{a}{q};x) - T(\frac{a}{q}; \frac{x}{h} -1)$.) Write
  \begin{align*}
   T(\alpha;N) - \frac{S(q,a)}{kq} U(\vartheta;N) = \sum_{n\leq N} \underbrace{\Bigg(n^{\omega}\, \frac{\mathbbm{1}_{\N^k}(n)}{n^{1/k}}\, e\bigg(\frac{na}{q}\bigg) - \frac{S(q,a)}{kq} n^{\omega-1} \Bigg)}_{=:\, D(n)} e(n\vartheta)
  \end{align*}
  By \eqref{gxab1}, and the fact that $\sum_{n\leq t} n^{\omega-1} = t^{\omega}/\omega + O(1+t^{\omega-1})$, for $t\leq N$ we have 
  \[ \sum_{n\leq t} D(n) = O(L^{k\omega+1} + L^2 t^{\omega-\frac{1}{k}} + t^{\omega-1} + 1) = O(N^{\omega-5\nu}) \]
  uniformly for $q\leq L = N^{\nu}$, by our choice of $\nu = \nu(\omega)$ in \eqref{nuchoice}. By partial summation we conclude, since $|\vartheta| \leq L/N = N^{-1+\nu}$,
  \begin{align*}
   T(\alpha;N) - \frac{S(q,a)}{kq} U(\vartheta;N) &= \bigg(\sum_{n\leq N} D(n)\bigg) e(N\vartheta) - 2\pi i \vartheta \int_{1}^{N} \bigg(\sum_{n\leq t} D(n)\bigg) e(t\vartheta)\,\mathrm{d}t \\
   &\ll \bigg|\sum_{n\leq N} D(n)\bigg| + N|\vartheta| \max_{t\leq N} \bigg(\sum_{n\leq t} D(n)\bigg) \\
   &= O(N^{\omega-4\nu}). \qedhere
  \end{align*}
 \end{proof}
 
 Next, we prove a lemma ensures that the singular series $\mathfrak{S}_{k,h}(N)$ converges absolutely and is bounded away from $0$ and $\infty$ for $h\geq 2H_0+1$.
 
 \begin{lem}[Singular series]\label{SingSerlem}
  Let $S(q,a)$ be as in Lemma \ref{war-est2}, and let
  \[ \mathfrak{S}_{k,h}(N,Q) := \sum_{q\leq Q} \sum_{\substack{a=1\\ (a,q)=1}}^{q} \frac{S(q,a)^{h}}{q^h} e\bigg(-\frac{Na}{q}\bigg), \quad \mathfrak{S}_{k,h}(N) = \lim_{Q\to\infty} \mathfrak{S}_{k,h}(N,Q). \]
  Then, for $h\geq 2H_0+1$ we have 
  \[ \mathfrak{S}_{k,h}(N,Q) = (1 + O(Q^{-1/k}))\, \mathfrak{S}_{k,h}(N) \quad\text{as}\quad Q\to\infty. \]
  Moreover, $\mathfrak{S}_{k,h}(N) \asymp 1$.
 \end{lem}
 \begin{proof}
  Let $q\geq 1$, $1\leq a\leq q$ with $(a,q)=1$. By \cite[Theorem 4.2]{vaughan97}, we have $S(q,a) \ll q^{1 - 1/k}$. Thus, since $H_0\geq k$, and hence $h\geq 2k+1$, we have
  \begin{align*}
   |\mathfrak{S}_{k,h}(N) - \mathfrak{S}_{k,h}(N,Q)| &\leq \sum_{q>Q} \sum_{\substack{a=1\\ (a,q)=1}}^{q} \bigg|\frac{S(q,a)^{h}}{q^{h}} \bigg| \\
   &\ll \sum_{q>Q} (q^{-1/k})^h \, q \\
   &\ll \sum_{q>Q} q^{-1 -1/k} \asymp Q^{-1/k}.
  \end{align*}
  In particular, this shows $\mathfrak{S}_{k,h}(N) \ll 1$.
  
  By \cite[Theorem 4.6]{vaughan97}, $\mathfrak{S}_{k,\ell}(N) \gg 1$ for $\ell \geq 5$ when $k=2$, for $\ell\geq 4k$ when $k>2$ is a power of $2$, and for $\ell\geq \frac{3}{2}k$ when $k\geq 2$ otherwise. In all cases, this holds true for $\ell \geq 2H_0+1$.
 \end{proof}
 
 To work with $U$, $U^{\sharp}$, we need the following bounds.
 
 \begin{lem}\label{est4}
  For $-\frac{1}{2}\leq \vartheta \leq \frac{1}{2}$, $\vartheta\neq 0$, we have
  \[ |U(\vartheta;N)| \ll N^{\omega-1}|\vartheta|^{-1} + |\vartheta|^{-\omega}, \qquad |U^{\sharp}(\vartheta;N)| \ll N^{\omega-1}|\vartheta|^{-1}, \]
 \end{lem}
 \begin{proof}
  We have $|U(\vartheta;N)| \leq U(0;N) = N^{\omega}/\omega + O(1+N^{\omega-1})$. On the other hand, let $M := \min\{N,\lfloor |\vartheta|^{-1}\rfloor\}$, so that 
  \begin{equation}
   |U(\vartheta;N)| \leq \sum_{n\leq M} n^{\omega-1} + \bigg|  \sum_{n=M+1}^{N} n^{\omega-1} e(n\vartheta) \bigg| \ll M^{\omega} + \bigg| \sum_{n=M+1}^{N} n^{\omega-1} e(n\vartheta) \bigg|. \label{unM}
  \end{equation}
  Let $E(x):= \sum_{n\leq x} e(n\vartheta)$. Since $|1-e(\vartheta)| = |e(-\vartheta/2) - e(\vartheta/2)| = 2|\sin \pi\vartheta| \geq 2|\vartheta|$ for $\vartheta$ in the range considered, we have
  \begin{align*}
   |E(x)| = \bigg|\sum_{n\leq x} e(n\vartheta)\bigg| = \bigg|\frac{1-e((\lfloor x \rfloor+1) \vartheta)}{1-e(\vartheta)} \bigg| \leq |\vartheta|^{-1}.
  \end{align*}
  By using that $E(x) \ll |\vartheta|^{-1}$, we obtain from partial summation that
  \begin{align*}
   \sum_{n=M+1}^{N} n^{\omega-1} e(n\vartheta) &= N^{\omega-1} E(N) - M^{\omega-1} E(M) - (\omega-1)\int_{M}^{N} E(t) t^{\omega-2}\, \mathrm{d}t \\
   &\ll (N^{\omega-1} + M^{\omega-1}) \,|\vartheta|^{-1} \ll N^{\omega-1}|\vartheta|^{-1} +  M^{\omega},
  \end{align*}
  so the result follows from \eqref{unM}.
  
  Similarly, for $U^{\sharp}$, we have
  \begin{align*}
   U^{\sharp}(\vartheta; N) &= \sum_{N/h \leq n \leq N} n^{\omega-1} e(n\vartheta) \\
   &= N^{\omega-1} E(N) - \bigg(\frac{N}{h}\bigg)^{\omega-1} E(N/h) - (\omega-1)\int_{N/h}^{N} E(t) t^{\omega-2}\, \mathrm{d}t \\
   &\ll N^{\omega-1} |\vartheta|^{-1}. \qedhere
  \end{align*}
 \end{proof}
 
 With this, we can now substitute $T$, $T^{\sharp}$ by $U$, $U^{\sharp}$ in the integral of Proposition \ref{majorwar}.
 
 \begin{lem}\label{war-quasiprop}
  For $1\leq j\leq h$, we have
  \begin{align*}
   \int_{\mathfrak{M}} T^{\sharp}&(\alpha;N)^{j}\,T(\alpha;N)^{h-j}\,e(-N\alpha)\,\mathrm{d}\alpha \\
   &= (1+O(N^{-\nu/k}))\, \mathfrak{S}_{k,h}(N) \frac{1}{k^h}\int_{0}^{1} U^{\sharp}(\vartheta;N)^{j}\, U(\vartheta; N)^{h-j}\,e(-N\vartheta)\,\mathrm{d}\vartheta + O(N^{h\omega-1 - \mu}),
  \end{align*}
  where $\mu := \nu\min\{1,(h-1)\omega\}$.
 \end{lem}
 \begin{proof}
  Since $|U(\vartheta;N)| \leq U(0;N) = N^{\omega}/\omega + O(1 + N^{\omega-1})$, Lemma \ref{war-est2} implies that
  \[ T(\alpha; N)^{h} = \frac{S(q,a)^{h}}{k^h q^{h}}U(\vartheta; N)^h + O(N^{h\omega - 4\nu}), \]
  and similarly for $T^{\sharp}$, $U^{\sharp}$. Using that $\int_{\mathfrak{M}(q,a)}\mathrm{d}\alpha = 2L/N$, this leads to the estimate
  \begin{align}
   \int_{\mathfrak{M}}\, &T^{\sharp}(\alpha;N)^{j}\,T(\alpha;N)^{h-j}\, e(-N\alpha)\,\mathrm{d}\alpha \nonumber \\
   &= \sum_{q\leq L}\sum_{\substack{a=1\\ (a,q)=1}}^{q} \int_{\mathfrak{M}(q,a)} T^{\sharp}(\alpha;N)^{j}\,T(\alpha;N)^{h-j}\, e(-N\alpha)\,\mathrm{d}\alpha \nonumber \\
   &= \sum_{q\leq L}\sum_{\substack{a=1\\ (a,q)=1}}^{q} \frac{S(q,a)^h}{k^h q^{h}} \int_{\mathfrak{M}(q,a)} U^{\sharp}\bigg(\alpha - \frac{a}{q}; N\bigg)^j\, U\bigg(\alpha - \frac{a}{q}; N\bigg)^{h-j} e(-N\alpha)\,\mathrm{d}\alpha \nonumber\\
   &\hspace{+28em}+ O(L^3 N^{h\omega - 1 -4\nu}) \nonumber \\
   &= \underbrace{\sum_{q\leq L} \sum_{\substack{a=1\\ (a,q)=1}}^{q} \frac{S(q,a)^{h}}{q^{h}} e\bigg(-\frac{Na}{q}\bigg)}_{=\,\mathfrak{S}_{k,h}(N,L)}\, \frac{1}{k^h}\int_{-L/N}^{L/N} U^{\sharp}(\vartheta; N)^j\, U(\vartheta; N)^{h-j}\,e(-N\vartheta) \,\mathrm{d}\vartheta \nonumber \\
   &\hspace{+29em}+ O(N^{h\omega-1-\nu}) \nonumber \\
   &= (1+O(N^{-\nu/k}))\, \mathfrak{S}_{k,h}(N) \frac{1}{k^h}\int_{-L/N}^{L/N} U^{\sharp}(\vartheta; N)^{j}\, U(\vartheta; N)^{h-j}\,e(-N\vartheta) \,\mathrm{d}\vartheta + O(N^{h\omega-1-\nu}), \nonumber
  \end{align}
  where the last line follows from Lemma \ref{SingSerlem}. For $|\vartheta| \geq N^{-1}$, Lemma \ref{est4} gives $|U(\vartheta;N)| \ll N^{\omega-1}|\vartheta|^{-1} + |\vartheta|^{-\omega} \ll N^{\omega-\min\{1,\omega\}}|\vartheta|^{-\min\{1,\omega\}}$. Hence,
  \begin{align}
   \bigg|\int_{[-\frac{1}{2},-\frac{L}{N}]\cup [\frac{L}{N},\frac{1}{2}]} U^{\sharp}(\vartheta;N)^{j} \,U(\vartheta;N&)^{h-j}\, e(-N\vartheta)\,\mathrm{d}\vartheta\bigg| \nonumber \\
   &\leq \int_{[-\frac{1}{2},-\frac{L}{N}]\cup [\frac{L}{N},\frac{1}{2}]} |U^{\sharp}(\vartheta;N)|^{j}\,|U(\vartheta;N)|^{h-j}\, \mathrm{d}\vartheta \nonumber \\
   &\ll N^{j(\omega-1)+(h-j)(\omega-\min\{1,\omega\})} \int_{L/N}^{1/2} \vartheta^{-j-(h-j)\min\{1,\omega\}} \,\mathrm{d}\vartheta \nonumber \\
   &\ll \frac{N^{h\omega-1}}{L^{j+(h-j)\min\{1,\omega\}-1}} \ll \frac{N^{h\omega-1}}{L^{(h-1)\min\{1,\omega\}}}. \label{estUshU} 
  \end{align}
  Since $\mathfrak{S}_{k,h}(N)\asymp 1$ and $L=N^{\nu}$, this finishes the proof.
 \end{proof}
 
 Lemma \ref{war-quasiprop} reduces Proposition \ref{majorwar} to the calculation of the integral in $U$, $U^{\sharp}$ that appears in its statement, which is done in the following lemma: 
 
 \begin{lem}[Singular integral]\label{sgint}
  We have
  \[ \sum_{j=1}^{h} (-1)^{j+1}\binom{h}{j} \int_{0}^{1} U^{\sharp}(\vartheta;N)^{j}\, U(\vartheta; N)^{h-j}\,e(-N\vartheta)\,\mathrm{d}\vartheta = (1+ O(N^{-\min\{1,\omega\}})) \frac{\Gamma(\omega)^{h}}{\Gamma(h\omega)} N^{h\omega - 1}. \]
 \end{lem}
 \begin{proof}
  By the same argument used to obtain \eqref{sothat1}, we deduce from \eqref{defU} that
  \begin{align}
   \sum_{j=1}^{h} (-1)^{j+1}\binom{h}{j} \int_{0}^{1} U^{\sharp}(\vartheta;N)^{j}\, U(\vartheta; N)^{h-j}\,e(-N\vartheta)\,\mathrm{d}\vartheta &= \int_{0}^{1} U(\vartheta;N)^h\, e(-N\vartheta)\,\mathrm{d}\vartheta \nonumber \\
   &= \sum_{\substack{x_1,\ldots,x_h \in \N \\ x_1+\cdots+x_h = N}} (x_1\cdots x_h)^{\omega - 1}. \nonumber
  \end{align}
  The evaluation of this sum for $\omega>0$ can be found, for instance, in \cite[Lemma 3.5]{taf25}.
 \end{proof}

 \begin{proof}[Proof of Proposition \ref{majorwar}]
  By Lemmas \ref{war-quasiprop} and \ref{sgint}, with $\mu = \nu\min\{1,(h-1)\omega\}$, we have
  \begin{align*}
   \sum_{j=1}^{h}\, &(-1)^{j+1} \binom{h}{j} \int_{\mathfrak{M}} T^{\sharp}(\alpha;N)^{j}\, T(\alpha;N)^{h-j}\,e(-N\alpha)\,\mathrm{d}\alpha \\
   &= (1+O(N^{-\nu/k}))\,\mathfrak{S}_{k,h}(N) \frac{1}{k^h}\sum_{j=1}^{h} (-1)^{j+1} \binom{h}{j}\int_{0}^{1} U^{\sharp}(\vartheta;N)^{j}\, U(\vartheta; N)^{h-j}\,e(-N\vartheta)\,\mathrm{d}\vartheta \\
   &\hspace{+29.5em}+ O(N^{h\omega-1-\mu}) \\
   &= (1+O(N^{-\nu/k}))\,(1+ O(N^{-\min\{1,\omega\}}))\, \mathfrak{S}_{k,h}(N) \frac{1}{k^h}\frac{\Gamma(\omega)^{h}}{\Gamma(h\omega)} N^{h\omega -1} + O(N^{h\omega-1-\mu}),
  \end{align*}
  Since $\mathfrak{S}_{k,h}(N)\ll 1$ by Lemma \ref{SingSerlem}, rearranging the error terms concludes the proof.
 \end{proof}

%%%%%%%%%%%%
\subsection{Minor arcs}\label{minarcwar}
 We now turn to the minor arcs $\mathfrak{m}$, as defined in the setup for Theorem \ref{wolem}. Together with \eqref{sothat1} and Proposition \ref{majorwar}, the next result directly implies Theorem \ref{wolem}.
 
 \begin{prop}[Minor arcs]\label{minorwar}
  We have
  \[ \sum_{j=1}^{h} \bigg|\int_{\mathfrak{m}} T^{\sharp}(\alpha; N)^{j}\,T(\alpha;N)^{h-j}\, e(-N\alpha) \,\mathrm{d}\alpha\bigg| \ll N^{h\omega - 1 - \eta'}, \]
  where $\eta' := 2^{-k}\nu \min\{\frac{1}{h},\,\frac{h-2H_0}{h} -2H_0(h-1)\delta\}$.
 \end{prop}

 Define the function
 \begin{equation*}
  g(\alpha;x) := \sum_{n\leq x} \mathbbm{1}_{\N^k}(n)\, e(n\alpha). %\label{funcg}
 \end{equation*}
 We bound the value of $g(\alpha;x)$ for $\alpha\in\mathfrak{m}$ using Weyl's inequality \cite[Theorem 4.3]{nathanson96}. 

 \begin{lem}[Weyl's inequality]\label{weylineq}
  Let $\alpha\in[0,1]$ be such that $\|\alpha - \frac{a}{q}\| \leq q^{-2}$, where $a$, $q$ are integers with $1\leq q\leq x$ and $(a,q)=1$. Then,
  \[ g(\alpha;x) \ll (q^{-1} + x^{-1/k} + q x^{-1})^{2^{1-k}} x^{1/k + o(1)} \]
 \end{lem}
 
 Modern versions of this inequality follow from current Vinogradov mean value estimates, with the exponent $2^{1-k}$ replaced by $(k(k-1))^{-1}$ \cite[Theorem 5]{bou17}. This refinement, however, is not needed for our purposes.

 \begin{cor}\label{war-condv}
  $\displaystyle \sup_{\alpha\in \mathfrak{m}} |g(\alpha;N)| \ll N^{\frac{1}{k} - 2^{1-k}\nu + o(1)}$.
 \end{cor}
 \begin{proof}
  By Dirichlet's theorem \cite[Theorem 4.1]{nathanson96}, for any $\alpha\in [0,1]$ there exists $1\leq q\leq N/L$ and $1\leq a\leq q$ with $(a,q)=1$ such that
  \[ \bigg\|\alpha - \frac{a}{q} \bigg\| \leq \frac{L}{qN} \leq \min\bigg(\frac{L}{N},\frac{1}{q^2}\bigg). \]
  Since $\alpha\in \mathfrak{m}$, we must have $q>L$ (otherwise, $\alpha\in \mathfrak{M}_{N}(q,a)$ by definition), so $L< q\leq N/L$. Hence, by Lemma \ref{weylineq},
  \begin{align*}
   g(\alpha;N) &\ll (q^{-1} + N^{-1/k} + qN^{-1})^{2^{1-k}} N^{1/k + o(1)} \\
   &\ll (L^{-1} + N^{-1/k})^{2^{1-k}} N^{1/k + o(1)} \ll \frac{N^{1/k + o(1)}}{L^{2^{1-k}}} = N^{\frac{1}{k} - 2^{1-k}\nu + o(1)}. \qedhere
  \end{align*}
 \end{proof}
 
 To prove Proposition \ref{minorwar}, we need two auxiliary lemmas concerning the integral of $T^{\sharp}$.
 
 \begin{lem}\label{war-cntsols}
  For $\ell \geq 2H_0$, we have
  \[  \int_{0}^{1} |T(\alpha;x)|^{\ell}\,\mathrm{d}\alpha \ll x^{\max\{\ell\omega-1,\,0\} +o(1)},\qquad \int_{0}^{1} |T^{\sharp}(\alpha;x)|^{\ell}\,\mathrm{d}\alpha \ll x^{\ell\omega-1 +o(1)}. \]
 \end{lem}
 \begin{proof}
  This follows from Lemmas \ref{weighted-hua} and \ref{dyadic-hua} (see the Remark after Lemma \ref{war-hua}).
 \end{proof}
 
 \begin{lem}\label{war-bdH}
  For $h\geq 2H_0+1$, we have
  \[ \int_{\mathfrak{m}} |T^{\sharp}(\alpha; N)|^h\,\mathrm{d}\alpha \ll N^{h\omega-1 - 2^{1-k}\nu + o(1)}. \]
 \end{lem}
 \begin{proof}
  The argument of Corollary \ref{war-condv} applies uniformly to $[N/h,N]$, so 
  \[ \sup_{N/h\leq t\leq N} |g(\alpha;t)| \ll N^{\frac{1}{k} - 2^{1-k}\nu + o(1)} \]
  for $\alpha\in\mathfrak{m}$. Hence, for $\alpha\in\mathfrak{m}$, partial summation yields
  \begin{align}
   T^{\sharp}(\alpha; N) &= \frac{N^{\omega}}{N^{1/k}}\, g(\alpha;N) - \frac{(N/h)^{\omega}}{(N/h)^{1/k}} \, g(\alpha; N/h) - \int_{N/h}^{N} g(\alpha;t)\,\mathrm{d}\bigg(\frac{t^{\omega}}{t^{1/k}}\bigg) \nonumber \\
   &\ll \frac{N^{\omega}}{N^{1/k}}\, N^{\frac{1}{k} - 2^{1-k}\nu + o(1)} = N^{\omega - 2^{1-k}\nu + o(1)}. \label{bdsup}
  \end{align}
  Since $h\geq 2H_0+1$, it follows from Lemma \ref{war-cntsols} that
  \begin{align*}
   \int_{\mathfrak{m}} |T^{\sharp}(\alpha;N)|^h\,\mathrm{d}\alpha &\leq \bigg(\sup_{\alpha\in\mathfrak{m}} |T^{\sharp}(\alpha; N)|^{h-2H_0} \bigg)\,\int_{0}^{1} |T^{\sharp}(\alpha;N)|^{2H_0}\,\mathrm{d}\alpha \\
   &\ll (N^{\omega - 2^{1-k}\nu + o(1)})^{h-2H_0}\, N^{2H_0\omega - 1 + o(1)} \\
   &\ll N^{h\omega-1 - 2^{1-k}\nu + o(1)}. \qedhere
  \end{align*}
 \end{proof} 
 
 Up to here, we only needed to assume $\omega > 0$. Only in the final argument to bound the minor arc integral we will use that $\omega \geq 1/h - \delta$ for some $0 \leq \delta < (h-2H_0)/(2h(h-1)H_0)$.
 
 \begin{proof}[Proof of Proposition \ref{minorwar}]
  Assume first that $\omega \geq 1/h$, so $\max\{h\omega-1,\,0\} = h\omega-1$. Then, by Lemmas \ref{war-cntsols} and \ref{war-bdH}, H\"older's inequality yields
  \begin{align*}
   \sum_{j=1}^{h} \bigg|\int_{\mathfrak{m}} T^{\sharp}(\alpha; N)^{j}\,T(\alpha;N)^{h-j}\,& e(-N\alpha) \,\mathrm{d}\alpha\bigg| \\
   &\leq \sum_{j=1}^{h} \bigg(\int_{\mathfrak{m}} |T^{\sharp}(\alpha;N)|^h\,\mathrm{d}\alpha \bigg)^{j/h}\,\bigg(\int_{0}^{1} |T(\alpha;N)|^{h}\,\mathrm{d}\alpha\bigg)^{1-j/h} \\
   &\leq \sum_{j=1}^{h} (N^{h\omega - 1 - 2^{1-k}\nu + o(1)})^{j/h}\, (N^{h\omega-1 + o(1)})^{1-j/h} \\
   &\ll N^{h\omega-1 - 2^{1-k}\nu/h + o(1)},
  \end{align*}
  proving the first part.

  Now suppose that $1/h-\delta \leq \omega < 1/h$, and put $\theta:=1-(h-1)\omega$. Our hypothesis on $\delta$ ensures that
  \[ \theta \leq \frac{1}{h}+(h-1)\delta < \frac{1}{2H_0}. \]
  For $1\leq j\leq h$, another application of H\"older's inequality (cf. Wooley \cite[Eq. (2.7)]{woo04}) gives
  \begin{equation}\label{minarc-holder2}
   \bigg|\int_{\mathfrak{m}} T^{\sharp}(\alpha;N)^{j}\,T(\alpha;N)^{h-j}\,e(-N\alpha)\,\mathrm{d}\alpha\bigg| \ll \bigg(\sup_{\alpha\in\mathfrak{m}} |T^{\sharp}(\alpha;N)|\bigg)^{1-2H_0\theta}\, \Upsilon_1^{\theta}\, \Upsilon_2^{(j-1)\omega}\, \Upsilon_3^{(h-j)\omega},
  \end{equation}
  where
  \[ \Upsilon_1:=\int_{0}^{1} |T^{\sharp}(\alpha;N)|^{2H_0}\,\mathrm{d}\alpha,\quad
     \Upsilon_2:=\int_{0}^{1} |T^{\sharp}(\alpha;N)|^{1/\omega}\,\mathrm{d}\alpha,\quad
     \Upsilon_3:=\int_{0}^{1} |T(\alpha;N)|^{1/\omega}\,\mathrm{d}\alpha. \]
  By \eqref{bdsup} in Lemma \ref{war-bdH} we have $\sup_{\alpha\in\mathfrak{m}} |T^{\sharp}(\alpha;N)| \ll N^{\omega-2^{1-k}\nu+o(1)}$, and by Lemma \ref{war-cntsols} we have $\Upsilon_1 \ll N^{2H_0\omega-1+o(1)}$ and $\Upsilon_2, \Upsilon_3 \ll N^{o(1)}$ (since $(1/\omega)\omega-1=0$ and $1/\omega > h > 2H_0$ in the present range). Substituting these estimates into \eqref{minarc-holder2}, we obtain
  \[ \bigg|\int_{\mathfrak{m}} T^{\sharp}(\alpha;N)^{j}\,T(\alpha;N)^{h-j}\,e(-N\alpha)\,\mathrm{d}\alpha\bigg| \ll N^{h\omega-1- 2^{1-k}\nu(1-2H_0\theta) + o(1)}, \]
  which, since $1-2H_0\theta \geq \frac{h-2H_0}{h} - 2H_0(h-1)\delta > 0$, concludes the proof.
 \end{proof}

%%%%%%%%%%%%%%%%%%%%%%%%%%%%%%%%%%%%%%%%%
\section{Waring--Goldbach subbases: Theorem \ref{MTp1}}
 In this section we prove Theorem \ref{MTp1}. The argument parallels the proof of Theorem \ref{MT1}. The central analytic input is a weighted asymptotic formula for the Waring--Goldbach representation problem (Theorem \ref{MTp}), in the same spirit as Theorem \ref{wolem}. The circle method framework is standard, and the calculations run analogously to those of Section \ref{waring1}. We start by verifying that $\P^k$ satisfies the Hua-type hypothesis \eqref{hua-type} required by the general subbasis principle. %This follows from Lemma \ref{war-hua}.
 
 %SOLS IN PRIME POWERS << SOLS IN POWERS
 \begin{lem}[Mean value estimate]\label{condiv}
  Let $k\geq 1$, and let $H_0=H_0(k)$ be as in \eqref{h0k}. Then, for every integer $\ell \geq H_0$,
  \[ \int_{0}^{1} \bigg| \sum_{p\leq x} e(\alpha p^{k})\bigg|^{2\ell} \,\mathrm{d}\alpha \ll x^{2\ell - k}(\log x)^M \]
  for some constant $M = M(k) >0$. When $k\geq 5$, the factor $(\log x)^M$ can be dropped.
 \end{lem}
 \begin{proof}
  By the orthogonality of the $e(\alpha n)$, the integral $\int_{0}^{1} | \sum_{p\leq x} e(\alpha p^{k})|^{2\ell}\,\mathrm{d}\alpha$ counts the number of solutions to the equation 
  \begin{equation*}
   p_1^k + \cdots + p_{\ell}^k = q_1^k + \cdots + q_{\ell}^k
  \end{equation*}
  with $p_i,q_j \leq x$ primes. This is bounded from above by the number of solutions to $\sum_{i=1}^{\ell} n_i^k = \sum_{i=1}^{\ell} m_i^k$ with $n_i, m_i \leq x$ positive integers, hence
  \[ \int_{0}^{1} \bigg| \sum_{p\leq x} e(\alpha p^{k})\bigg|^{2\ell} \,\mathrm{d}\alpha \leq \int_{0}^{1} \bigg| \sum_{n\leq x} e(\alpha n^{k})\bigg|^{2\ell} \,\mathrm{d}\alpha. \]
  The result then follows from Lemma \ref{war-hua}. \qedhere
 \end{proof}
 
 \begin{xrem}
  From Lemma \ref{condiv}, by Parseval's identity, we have
  \[ \sum_{n\leq x} r_{\P^{k},\ell}(n)^{2} \leq \int_{0}^{1} \bigg| \sum_{p\leq x^{1/k}} e(\alpha p^{k})\bigg|^{2\ell} \,\mathrm{d}\alpha \ll x^{2\ell/k - 1}(\log x)^M, \]
  which, since $|\P^k\cap[0,x]| \asymp x^{1/k}/\log x$, is in the shape of \eqref{hua-type}.
 \end{xrem}
 
 By the prime number theorem, $|\P^k\cap[0,x]| = |\P\cap [0,x^{1/k}]| \sim kx^{1/k}/\log x$, so $\P^k$ satisfies \eqref{oregvar}. With Lemma \ref{condiv}, $\P^k$ satisfies \eqref{hua-type}. Thus, assuming Theorem \ref{MTp}, Theorem \ref{MTp1} follows from Corollary \ref{cor-tech2} applied to every $0\leq \kappa\leq h/k-1$.

%%%%%%%%%%%%%%%%%%%%%%%
\subsection{Setup for Theorem \ref{MTp}}
 We will now prove Theorem \ref{MTp}. Fix $k\geq 1$, and let $h\geq 2H_0+1$ be an integer, where $H_0 = H_0(k)$ is as in \eqref{h0k}. Let $0 < \delta < (h-2H_0)/2h(h-1)H_0$, let $\omega \geq 1/h - \delta$ be a real number, and let $N$ be a large integer. Define
 \begin{equation}
  T(\alpha;x) = \sum_{n\leq x} n^{\omega}\, \frac{\mathbbm{1}_{\P^k}(n)}{n^{1/k}}\,(\log n) \, e(\alpha n), \quad T^{\sharp}(\alpha;x) = \sum_{x/h\leq n\leq x} n^{\omega}\, \frac{\mathbbm{1}_{\P^k}(n)}{n^{1/k}}\, (\log n)\, e(\alpha n), \label{Tprime}
 \end{equation}
 so that, by the same argument as in Section \ref{waring1},
 \begin{align}
  \sum_{\substack{x_1,\ldots,x_h \in \P^k \\ x_1+\cdots+x_h = N}} (x_1\cdots x_h&)^{\omega-\frac{1}{k}} (\log x_1\cdots \log x_h) \nonumber \\
  &= \int_{0}^{1} \big(T(\alpha;N)^h - (T(\alpha;N)-T^{\sharp}(\alpha;N))^h\big)\, e(-N\alpha)\,\mathrm{d}\alpha \nonumber \\
  &= \sum_{j=1}^{h} (-1)^{j+1}\binom{h}{j} \int_{0}^{1} T^{\sharp}(\alpha;N)^{j}\,T(\alpha;N)^{h-j}\, e(-N\alpha)\,\mathrm{d}\alpha. \label{sothat}
 \end{align} 
 
 We once again divide $[0,1]$ into major arcs $\mathfrak{M}$ and minor arcs $\mathfrak{m}$, where the major arcs $\mathfrak{M}=\mathfrak{M}_N$ are defined this time as
 \[ \mathfrak{M} := \bigsqcup_{q\leq Q}\bigsqcup_{\substack{a=1 \\ (a,q)=1}}^{q} \mathfrak{M}(q,a), \quad \mathfrak{M}(q,a) := \bigg\{\alpha \in [0,1] ~\bigg|~ \bigg\|\alpha -\frac{a}{q}\bigg\| \leq \frac{Q}{N} \bigg\}  \]
 for a parameter $Q = Q(N) := (\log N)^{C}$, where $C>1$ is a large fixed constant. The minor arcs are defined as the complement, $\mathfrak{m} = \mathfrak{m}_N := [0,1]\setminus \mathfrak{M}$. By \eqref{sothat}, Theorem \ref{MTp} will follow from a major arc evaluation (Proposition \ref{sssi}) and a minor arc bound saving a power of $\log$ that becomes arbitrarily large as $C\to\infty$ (Proposition \ref{mARC}).

%%%%%%%%%%%%%%%%%%%%%%
\subsection{Major arcs}
 In this subsection, we will prove the following: 
 
 \begin{prop}[Major arcs]\label{sssi}
  We have
  \begin{align*}
   \sum_{j=1}^{h} (-1)^{j+1}\binom{h}{j} \int_{\mathfrak{M}} T^{\sharp}(\alpha;N)^{j}\, T(\alpha;\,N)^{h-j}\,&e(-N\alpha)\,\mathrm{d}\alpha \\
   &= \mathfrak{S}^{*}_{k,h}(N)\,\frac{\Gamma(\omega)^h}{\Gamma(h\omega)} \, N^{h\omega-1} + O_R\bigg(\frac{N^{h\omega-1}}{(\log N)^{R}}\bigg),
  \end{align*}
  where $\mathfrak{S}^{*}_{k,h}(N)$ is as in \eqref{singserp}, and $R(C) \to \infty$ as $C\to\infty$.
 \end{prop}

 Given $q\geq 1$, $1\leq a\leq q$, define
 \[ \pi(x;q,a) := \sum_{\substack{n\leq x\\ n\equiv a(q)}} \mathbbm{1}_{\P}(n). \]
 We start with the classical estimate for primes in arithmetic progressions due to Siegel and Walfisz (cf. Montgomery--Vaughan \cite[Corollary 11.21]{montvaug06}), and a weighted consequence of it.
  
 \begin{lem}[Siegel--Walfisz]\label{siewal}
  Uniformly for $q\leq (\log x)^C$, $1\leq a\leq q$ with $(a,q)=1$, we have
  \[ \pi(x;q,a) = \bigg(1 + O_{R}\bigg(\frac{1}{(\log x)^{R}}\bigg)\bigg)\frac{\li(x)}{\varphi(q)} \]
  for every $R>1$, where $\li(x) = \int_{2^{-}}^{x} (\log t)^{-1}\,\mathrm{d}t$.
 \end{lem}
  
 \begin{cor}\label{est1}
  Uniformly for $q\leq (\log x)^C$, $1\leq a \leq q$ with $(a,q)=1$, we have
  \[ \sum_{\substack{n\leq x\\ n\equiv a(q)}} n^{\omega}\,\frac{\mathbbm{1}_{\P^k}(n)}{n^{1/k}}\log n = \bigg(1+O_{R}\bigg(\frac{1}{(\log x)^R}\bigg)\bigg)\frac{P_k(q,a)}{\varphi(q)}\frac{x^{\omega}}{\omega} \]
  for every $R>1$, where $P_k(q,a) = \#\{1\leq r\leq q ~|~ r^k \equiv a\pmod{q}\}$.
 \end{cor}
 \begin{proof}
  Write $P = P_k(q,a)$, and let $1\leq r_1,\ldots,r_P\leq q$ be such that $r_i^k \equiv a\pmod{q}$. Since $\#\{n\leq x ~|~ n\in \P^{k},\ n\equiv a\pmod{q}\} = \sum_{i=1}^{P} \pi(x^{1/k};q,r_i)$, we have
  \begin{align*}
   \sum_{\substack{n\leq x\\ n\equiv a(q)}} n^{\omega}\,\frac{\mathbbm{1}_{\P^k}(n)}{n^{1/k}}\log n
   &= \sum_{i=1}^{P} \int_{2^{-}}^{x} t^{\omega-\frac{1}{k}}(\log t)\, \mathrm{d}\pi(t^{1/k};q,r_i) \\
   &= \frac{P}{\varphi(q)}\frac{x^{\omega}}{\omega} + \sum_{i=1}^{P} \int_{2^{-}}^{x} t^{\omega-\frac{1}{k}}(\log t)\, \mathrm{d}\bigg(\pi(t^{1/k};q,r_i) - \frac{\li(t^{1/k})}{\varphi(q)} \bigg).
  \end{align*}
  By partial summation, Lemma \ref{siewal} implies that
  \begin{align*}
   \int_{2^{-}}^{x} t^{\omega-\frac{1}{k}}&(\log t)\, \mathrm{d}\bigg(\pi(t^{1/k};q,r_i) - \frac{\li(t^{1/k})}{\varphi(q)} \bigg) \\
   &= \bigg(\pi(x^{1/k};q,r_i) - \frac{\li(x^{1/k})}{\varphi(q)} \bigg)\, x^{\omega-\frac{1}{k}}\log x \\
   &\hspace{8em} - \int_{2^{-}}^{x} \bigg(\pi(t^{1/k};q,r_i) - \frac{\li(t^{1/k})}{\varphi(q)} \bigg) \,t^{\omega-\frac{1}{k}-1}(1+(\omega-\tfrac{1}{k})\log t)\,\mathrm{d}t \\
   &\ll_{R} \frac{x^{\omega}}{(\log x)^{R}}, \end{align*}
  uniformly for $q\leq (\log x)^C$. Choosing $R$ larger by $C$ (and relabeling $R$) absorbs the factor $P\leq q$, concluding the proof.
 \end{proof}
 
 Recall the definitions of $U$ and $U^{\sharp}$ from Section \ref{waring1} (cf. \eqref{defU}):
 \begin{equation*}
  U(\vartheta; x) := \sum_{n\leq x} n^{\omega -1} e(n\vartheta), \qquad U^{\sharp}(\vartheta; x) := \sum_{x/h\leq n\leq x} n^{\omega -1} e(n\vartheta).
 \end{equation*}
 Note that the definitions of $T$, $T^{\sharp}$ in this section are different (tailored for prime powers), but for $U$, $U^{\sharp}$ they are the same. We now derive an approximation of $T$, $T^{\sharp}$ in terms of $U$, $U^{\sharp}$ that is analogous to that in Lemma \ref{war-est2}. 
 
 \begin{lem}\label{est2}
  Let $q\leq Q$, and $1\leq a\leq q$ with $(a,q)=1$. For $\alpha \in \mathfrak{M}(q,a)$, write $\vartheta := \alpha - \frac{a}{q}$. Then, for every $R>1$, we have
  \[ T(\alpha; N) = \frac{S^{*}(q,a)}{\varphi(q)} U(\vartheta; N) + O_R\bigg(\frac{N^{\omega}}{(\log N)^R}\bigg), \]
  \[ T^{\sharp}(\alpha; N) = \frac{S^{*}(q,a)}{\varphi(q)} U^{\sharp}(\vartheta; N) + O_R\bigg(\frac{N^{\omega}}{(\log N)^R}\bigg) \]
  uniformly in $q\leq Q$, where $S^{*}(q,a) = \sum_{\substack{1\leq r\leq q \\ (r,q)=1}} e(ar^k/q)$.
 \end{lem}
 \begin{proof}
  We prove the lemma for $T$; the argument for $T^{\sharp}$ is analogous.
  
  By Lemma \ref{est1}, for $x\leq N$ we have
  \begin{align}
   T\bigg(\frac{a}{q}; x\bigg) &= \sum_{n\leq x} n^{\omega}\,\frac{\mathbbm{1}_{\P^k}(n)}{n^{1/k}}\,(\log n)\, e\bigg(\frac{na}{q}\bigg) \nonumber \\
   &= \sum_{\substack{r=1 \\ (r,q)=1}}^{q} \Bigg(\sum_{\substack{n\leq x\\ n\equiv r(q)}} n^{\omega}\,\frac{\mathbbm{1}_{\P^k}(n)}{n^{1/k}} \log n\Bigg) e\bigg(\frac{ra}{q}\bigg) + O(q) \nonumber \\
   &= \Bigg(\frac{1}{\varphi(q)} \sum_{\substack{r=1 \\ (r,q)=1}}^{q} P_k(q,r)\,e\bigg(\frac{ra}{q} \bigg) \Bigg) \frac{x^{\omega}}{\omega} + O_R\bigg(\frac{N^{\omega}}{(\log N)^R}\bigg) \nonumber \\
   &= \frac{S^{*}(q,a)}{\varphi(q)} \frac{x^{\omega}}{\omega} + O_R\bigg(\frac{N^{\omega}}{(\log N)^R}\bigg). \label{gxab}
  \end{align}  
  (Note that $T^{\sharp}(\frac{a}{q};x) = T(\frac{a}{q};x) - T(\frac{a}{q}; \frac{x}{h} -1)$.) Write
  \begin{align*}
   T(\alpha;N) - \frac{S^{*}(q,a)}{\varphi(q)} U(\vartheta;N) = \sum_{n\leq N} \underbrace{\Bigg(n^{\omega}\, \frac{\mathbbm{1}_{\P^k}(n)}{n^{1/k}}\,(\log n)\, e\bigg(\frac{na}{q}\bigg) - \frac{S^{*}(q,a)}{\varphi(q)} n^{\omega-1} \Bigg)}_{=:\, D(n)} e(n\vartheta)
  \end{align*}
  By \eqref{gxab}, and the fact that $\sum_{n\leq t} n^{\omega-1} = t^{\omega}/\omega  + O(1 + t^{\omega-1})$, for $t\leq N$ we have $\sum_{n\leq t} D(n) = O_{R}(N^{\omega}/(\log N)^{R})$. By partial summation, since $|\vartheta| \leq Q/N = (\log N)^{C}/N$,
  \begin{align*}
   T(\alpha;N) - \frac{S^{*}(q,a)}{\varphi(q)} U(\vartheta;N) &= \bigg(\sum_{n\leq N} D(n)\bigg) e(N\vartheta) - 2\pi i \vartheta \int_{1}^{N} \bigg(\sum_{n\leq t} D(n)\bigg) e(t\vartheta)\,\mathrm{d}t \\
   &\ll \bigg|\sum_{n\leq N} D(n)\bigg| + N|\vartheta| \max_{t\leq N} \bigg(\sum_{n\leq t} D(n)\bigg) \\
   &= O_{R}\bigg(\frac{N^{\omega}}{(\log N)^{R}}\bigg). \qedhere
  \end{align*}
 \end{proof}
 
 Next, we show that the singular series $\mathfrak{S}_{k,h}^{*}(N)$ converges absolutely and is bounded away from $0$ and $\infty$ for $N\equiv h\pmod{K(k)}$, where $K(k)$ is as in \eqref{Kkval}.
 
 \begin{lem}[Singular series]\label{SingSerPlem}
  Let $S^{*}(q,a)$ be as in Lemma \ref{est2}, and let
  \[ \mathfrak{S}^{*}_{k,h}(N,L) := \sum_{q\leq L} \sum_{\substack{a=1\\ (a,q)=1}}^{q} \frac{S^{*}(q,a)^{h}}{\varphi(q)^h} e\bigg(-\frac{Na}{q}\bigg), \quad \mathfrak{S}^{*}_{k,h}(N) = \lim_{L\to\infty} \mathfrak{S}^{*}_{k,h}(N,L). \]
  Then, for $h\geq 2H_0+1$ we have
  \[ \mathfrak{S}^{*}_{k,h}(N,L) = (1 + O_{\eps}(L^{-1/2+\eps}))\, \mathfrak{S}^{*}_{k,h}(N) \quad\text{as}\quad L\to\infty, \]
  Moreover, $\mathfrak{S}^{*}_{k,h}(N) \ll 1$ for all $N$, and $\mathfrak{S}^{*}_{k,h}(N) \gg 1$ for $N\equiv h\pmod{K(k)}$.
 \end{lem}
 \begin{proof}
  Let $q\geq 1$, $1\leq a\leq q$ with $(a,q)=1$. When $k=1$, we have $S^{*}(q,a)=\mu(q)$. When $k\geq 2$, by \cite[Lemma 8.5]{hua65} we have $S^{*}(q,a) \ll_{\eps} q^{1/2+\eps}$. Thus, since $\varphi(q) \gg_{\eps} q^{1-\eps}$ we have
  \begin{align*}
   |\mathfrak{S}_{k,h}^{*}(N) - \mathfrak{S}_{k,h}^{*}(N,L)|
   &\leq \sum_{q>L} \sum_{\substack{a=1\\ (a,q)=1}}^{q} \bigg|\frac{S^{*}(q,a)^{h}}{\varphi(q)^{h}} \bigg| \\
   &\ll_{\eps}
   \begin{cases}
    \displaystyle \sum_{q>L} \varphi(q)^{1-h} \ll_{\eps} \sum_{q>L} q^{1-h+\eps} \asymp L^{2-h+\eps}, & k=1, \\[1.0ex]
    \displaystyle \sum_{q>L} \varphi(q)\,\frac{q^{h/2+\eps}}{\varphi(q)^{h}}
    \ll_{\eps} \sum_{q>L} q^{1-h/2+\eps} \asymp L^{2-h/2+\eps}, & k\geq 2.
   \end{cases}
  \end{align*}
  Since $H_0(1) =1$ and $H_0(k) \geq 2$ for $k\geq 2$, and $h\geq 2H_0+1$, the stated error term follows. In particular, this shows $\mathfrak{S}_{k,h}^{*}(N) \ll 1$.
  
  By \cite[Lemma 8.12]{hua65}, $\mathfrak{S}_{k,\ell}^{*}(N) \gg 1$ for $N\equiv \ell\pmod{K(k)}$ whenever $\ell \geq 3$ for $k=1$, $\ell \geq 5$ for $k=2$, and $\ell\geq 3k$ for $k\geq 3$. In all cases, this holds true for $\ell \geq 2H_0+1$.
 \end{proof}
 
 We are now able to substitute $T$, $T^{\sharp}$ by $U$, $U^{\sharp}$ in the integral of Proposition \ref{sssi}.
 
 \begin{lem}\label{quasiprop}
  For $1\leq j\leq h$, for every $\eps > 0$ we have
  \begin{align*}
   \int_{\mathfrak{M}} T^{\sharp}(\alpha;N)^{j}\, &T(\alpha;N)^{h-j}\,e(-N\alpha)\,\mathrm{d}\alpha \\
   &= (1+O_{\eps}((\log N)^{-C/2+\eps}))\,\mathfrak{S}^{*}_{k,h}(N)\int_{0}^{1} U^{\sharp}(\vartheta;N)^{j}\, U(\vartheta; N)^{h-j}\,e(-N\vartheta)\,\mathrm{d}\vartheta \\
   &\hspace{+19.5em}+ O\bigg(\frac{N^{h\omega-1}}{(\log N)^{C(h-1)\min\{1,\omega\}}}\bigg).
  \end{align*}
 \end{lem}
 \begin{proof}
  Since $|U(\vartheta;N)| \leq U(0;N) = N^{\omega}/\omega + O(1+N^{\omega-1})$, Lemma \ref{est2} implies that
  \[ T(\alpha; N)^{h} = \frac{S^{*}(q,a)^{h}}{\varphi(q)^{h}} U(\vartheta; N)^h + O_R\bigg(\frac{N^{h\omega}}{(\log N)^R}\bigg), \]
  and similarly for $T^{\sharp}$, $U^{\sharp}$. Using that $\int_{\mathfrak{M}(q,a)}\mathrm{d}\alpha = 2Q/N$, this leads to the estimate
  \begin{align}
   \int_{\mathfrak{M}}\, &T^{\sharp}(\alpha;N)^{j}\,T(\alpha;N)^{h-j}\, e(-N\alpha)\,\mathrm{d}\alpha \nonumber \\
   &= \sum_{q\leq Q}\sum_{\substack{a=1\\ (a,q)=1}}^{q} \int_{\mathfrak{M}(q,a)} T^{\sharp}(\alpha;N)^{j}\,T(\alpha;N)^{h-j}\, e(-N\alpha)\,\mathrm{d}\alpha \nonumber \\
   &= \sum_{q\leq Q}\sum_{\substack{a=1\\ (a,q)=1}}^{q} \frac{S^{*}(q,a)^h}{\varphi(q)^{h}} \int_{\mathfrak{M}(q,a)} U^{\sharp}\bigg(\alpha - \frac{a}{q}; N\bigg)^j\, U\bigg(\alpha - \frac{a}{q}; N\bigg)^{h-j} e(-N\alpha)\,\mathrm{d}\alpha \nonumber\\
   &\hspace{+26.7em}+ O_R\bigg(\frac{Q^3 N^{\omega h-1}}{(\log N)^{R}}\bigg) \nonumber \\
   &= \underbrace{\sum_{q\leq Q} \sum_{\substack{a=1\\ (a,q)=1}}^{q} \frac{S^{*}(q,a)^{h}}{\varphi(q)^{h}} e\bigg(-\frac{Na}{q}\bigg)}_{=\,\mathfrak{S}^{*}_{k,h}(N,Q)} \int_{-Q/N}^{Q/N} U^{\sharp}(\vartheta; N)^j\, U(\vartheta; N)^{h-j}\,e(-N\vartheta) \,\mathrm{d}\vartheta \nonumber \\
   &\hspace{+27em}+ O_R\bigg(\frac{N^{\omega h-1}}{(\log N)^{R}}\bigg) \nonumber \\
   &= (1+O_{\eps}((\log N)^{-C/2+\eps}))\,\mathfrak{S}^{*}_{k,h}(N) \int_{-Q/N}^{Q/N} U^{\sharp}(\vartheta; N)^{j} U(\vartheta; N)^{h-j}\,e(-N\vartheta) \,\mathrm{d}\vartheta \nonumber \\
   &\hspace{+27em}+ O_R\bigg(\frac{N^{\omega h-1}}{(\log N)^{R}}\bigg), \nonumber
  \end{align}
  where the last line follows from Lemma \ref{SingSerPlem}, and $R$ is chosen sufficiently large in terms of $C$. The same argument used to deduce \eqref{estUshU} at the end of Lemma \ref{war-quasiprop} yields
  \[ \bigg|\int_{[-\frac{1}{2},-\frac{Q}{N}]\cup [\frac{Q}{N},\frac{1}{2}]} U^{\sharp}(\vartheta;N)^{j}\, U(\vartheta;N)^{h-j}\, e(-N\vartheta)\,\mathrm{d}\vartheta\bigg| \ll \frac{N^{h\omega-1}}{Q^{(h-1)\min\{1,\omega\}}}. \]
  Since $\mathfrak{S}_{k,h}^{*}(N)\ll 1$, this finishes the proof.
 \end{proof}

 \begin{proof}[Proof of Proposition \ref{sssi}]
  Since both $C/2-\eps \to \infty$ and $C(h-1)\min\{1,\omega\} \to \infty$ as $C\to \infty$, Lemma \ref{quasiprop} together with Lemma \ref{sgint} (singular integral) yield
  \begin{align*}
   &\sum_{j=1}^{h} (-1)^{j+1} \binom{h}{j} \int_{\mathfrak{M}} T^{\sharp}(\alpha;N)^{j}\, T(\alpha;N)^{h-j}\,e(-N\alpha)\,\mathrm{d}\alpha \\
   &= \bigg(1+O_R\bigg(\frac{1}{(\log N)^{R}}\bigg)\bigg)\mathfrak{S}^{*}_{k,h}(N) \sum_{j=1}^{h} (-1)^{j+1} \binom{h}{j}\int_{0}^{1} U^{\sharp}(\vartheta;N)^{j}\, U(\vartheta; N)^{h-j}\,e(-N\vartheta)\,\mathrm{d}\vartheta \\
   &\hspace{+30em}+ O_R\bigg(\frac{N^{h\omega-1}}{(\log N)^{R}}\bigg) \\
   &= \bigg(1+O_R\bigg(\frac{1}{(\log N)^{R}}\bigg)\bigg)(1+ O(N^{-\min\{1,\omega\}}))\, \mathfrak{S}^{*}_{k,h}(N)\, \frac{\Gamma(\omega)^{h}}{\Gamma(h\omega)}\, N^{h\omega -1} +O_R\bigg(\frac{N^{h\omega-1}}{(\log N)^{R}}\bigg),
  \end{align*}
  where $R = R(C) \to \infty$ as $C\to \infty$. Since $\mathfrak{S}^{*}_{k,h}(N)\ll 1$ by Lemma \ref{SingSerPlem}, rearranging the error terms concludes the proof.
 \end{proof}

%%%%%%%%%%%%%%%%%%%%%%%%%%%%%%%%%%%%%%%%%% 
\subsection{Minor arcs}\label{minarc}
 We now turn to the minor arcs $\mathfrak{m}$, as defined in the setup for Theorem \ref{MTp}. Together with \eqref{sothat} and Proposition \ref{sssi}, the next result directly implies Theorem \ref{MTp}.
 
 \begin{prop}[Minor arcs]\label{mARC}
  We have
  \[ \sum_{j=1}^{h} \bigg|\int_{\mathfrak{m}} T^{\sharp}(\alpha; N)^{j}\,T(\alpha;N)^{h-j}\, e(-N\alpha) \,\mathrm{d}\alpha\bigg| \ll_R \frac{N^{h\omega-1}}{(\log N)^{R}}, \]
  where $R = R(C) \to \infty$ as $C\to\infty$.
 \end{prop}

 Define the function
 \begin{equation*}
  g(\alpha;x) := \sum_{n\leq x} \mathbbm{1}_{\P^k}(n) e(n\alpha). %\label{funcg}
 \end{equation*}
  
 We start by bounding the value of $g(\alpha;x)$ for $\alpha\in\mathfrak{m}$. To this end, we employ a lemma due to Harman, in the form stated by Kumchev--Tolev \cite[Lemma 3.3]{kumtol05}.

 \begin{lem}[Harman \cite{har81}]\label{harman}
  Let $\alpha\in[0,1]$ be such that $\|\alpha - \frac{a}{q}\| \leq q^{-2}$, where $a$, $q$ are integers with $1\leq q\leq x$ and $(a,q)=1$. Then, there exists a constant $\sigma = \sigma(k) >0$ such that
  \[ g(\alpha;x) \ll (q^{-1} + x^{-1/2k} + q^{1/2}x^{-1})^{4^{1-k}} x^{1/k}(\log x)^{\sigma} \]
 \end{lem}

 \begin{cor}\label{condv}
  $\displaystyle \sup_{\alpha\in \mathfrak{m}} |g(\alpha;N)| \ll \frac{N^{1/k}}{(\log N)^{4^{1-k}C-\sigma}}$.
 \end{cor}
 \begin{proof}
  By Dirichlet's theorem (cf. Nathanson \cite[Theorem 4.1]{nathanson96}), for any $\alpha\in [0,1]$ there exists $1\leq q\leq N/Q$ and $1\leq a\leq q$ with $(a,q)=1$ such that
  \[ \bigg\|\alpha - \frac{a}{q} \bigg\| \leq \frac{Q}{qN} \leq \min\bigg(\frac{Q}{N},\frac{1}{q^2}\bigg). \]
  Since $\alpha\in \mathfrak{m}$, we must have $q>Q$ (otherwise, $\alpha\in \mathfrak{M}_{N}(q,a)$ by definition), so $Q< q\leq N/Q$. Hence, by Lemma \ref{harman}, since $Q=(\log N)^C$,
  \begin{align*}
   g(\alpha;N) &\ll \bigg(Q^{-1} + N^{-1/2k} + \bigg(\frac{N}{Q}\bigg)^{1/2}N^{-1}\bigg)^{4^{1-k}} N^{1/k}(\log N)^{\sigma} \\
   &\ll \frac{N^{1/k}(\log N)^{\sigma}}{Q^{4^{1-k}}} = \frac{N^{1/k}}{(\log N)^{4^{1-k}C-\sigma}}. \qedhere
  \end{align*}
 \end{proof}
 
 In order to prove Proposition \ref{mARC}, we need three auxiliary lemmas.
 
 \begin{lem}\label{cntsols}
  For $\ell \geq 2H_0$, we have
  \[ \int_{0}^{1} |T^{\sharp}(\alpha;x)|^{\ell}\,\mathrm{d}\alpha \ll x^{\ell\omega-1}(\log x)^{M+2H_0}. \]
  where $M$ is as in Lemma \ref{condiv}.
 \end{lem}
 \begin{proof}
  The integral $\int_{0}^{1} |T^{\sharp}(\alpha;x)|^{2H_0}\,\mathrm{d}\alpha$ counts the number of solutions to the Diophantine equation $\sum_{i=1}^{H_0} n_i = \sum_{i=1}^{H_0} m_i$ for $x/h\leq n_i,m_i\leq x$, $n_i,m_i\in \P^k$, with weights
  \begin{align*}
   (n_1\ldots n_{H_0} m_1\cdots m_{H_0})^{\omega-\frac{1}{k}}(\log n_1)\cdots (\log n_{H_0})(\log m_1)\cdots &(\log m_{H_0}) \ll \bigg(\frac{x^{\omega}}{x^{1/k}} \log x\bigg)^{2H_0}.
  \end{align*}
  By considering this underlying equation, Lemma \ref{condiv} yields the upper bound
  \[ \int_{0}^{1} |T^{\sharp}(\alpha;x)|^{2H_0}\,\mathrm{d}\alpha \ll \frac{x^{2H_0\omega}}{x^{2H_0/k}}(\log x)^{2H_0} \int_{0}^{1} |g(\alpha; x)|^{2H_0}\, \mathrm{d}\alpha \ll x^{2H_0 \omega - 1}(\log x)^{M+2H_0}. \]
  Using the trivial estimate
  \[ |T^{\sharp}(\alpha;x)| \leq T^{\sharp}(0;x) \ll x^{\omega-\frac{1}{k}}(\log x) \sum_{x/h\leq n \leq x} \mathbbm{1}_{\P^{k}}(n) \ll x^{\omega}, \]
  for every real $\ell\geq 2H_0$ we get
  \begin{equation*}
   \int_{0}^{1} |T^{\sharp}(\alpha;x)|^{\ell}\,\mathrm{d}\alpha \ll x^{(\ell-2H_0)\omega} \int_{0}^{1} |T^{\sharp}(\alpha;x)|^{2H_0}\,\mathrm{d}\alpha \ll x^{\ell\omega-1}(\log x)^{M+2H_0}. \qedhere
  \end{equation*}
 \end{proof}
 
 \begin{lem}\label{cntsols2}
  For $\ell \geq 2H_0$, we have
  \[ \int_{0}^{1} |T(\alpha;x)|^{\ell}\,\mathrm{d}\alpha \ll x^{\max\{\ell\omega-1,\,0\}}(\log x)^{\ell + M+2H_0}. \]
  where $M$ is as in Lemma \ref{condiv}.
 \end{lem}
 \begin{proof}
  Let $h\geq 2H_0+1$ be fixed. The absolute value of the series $T(\alpha;x)$ can be bounded by
  \begin{equation*}
   |T(\alpha;x)| \leq \sum_{j\leq \frac{\log x}{\log h}} |T^{\sharp}(\alpha; h^{-j} x)| + O(1).
  \end{equation*}
  Thus, by H\"older's inequality %with p = 1/ell, q = ell/(ell-1)
  and Lemma \ref{cntsols}, we have 
  \begin{align}
   \int_{0}^{1} |T(\alpha;x)|^{\ell}\,\mathrm{d}\alpha &\ll (\log x)^{\ell-1} \int_{0}^{1} \bigg(\sum_{j\leq \frac{\log x}{\log h}} |T^{\sharp}(\alpha; h^{-j} x)|^{\ell}\bigg)\,\mathrm{d}\alpha \nonumber \\
   &\ll (\log x)^{\ell} \max_{j\leq \frac{\log x}{\log h}} \int_{0}^{1} |T^{\sharp}(\alpha; h^{-j} x)|^{\ell} \,\mathrm{d}\alpha \nonumber \\
   &\ll (\log x)^{\ell+M+2H_0}\,\max_{j\leq \frac{\log x}{\log h}} (h^{-j}x)^{\ell\omega-1} \nonumber \\
   &\ll x^{\max\{\ell\omega -1,\, 0\}} (\log x)^{\ell + M + 2H_0}. \nonumber\qedhere
  \end{align}
 \end{proof}
 
 \begin{lem}\label{bdH}
  For $h \geq 2H_0+1$, we have
  \[ \int_{\mathfrak{m}} |T^{\sharp}(\alpha; N)|^h\,\mathrm{d}\alpha \ll_R \frac{N^{h\omega-1}}{(\log N)^{R}}, \]
  where $R = R(C) \to \infty$ as $C\to\infty$.
 \end{lem}
 \begin{proof}
  The argument of Corollary \ref{condv} applies uniformly to $[N/h,N]$, so 
  \[ \sup_{N/h\leq t\leq N} |g(\alpha;t)| \ll \frac{N^{1/k}}{(\log N)^{4^{1-k}C-\sigma}} \]
  for $\alpha\in\mathfrak{m}$. Hence, for $\alpha\in\mathfrak{m}$, partial summation yields
  \begin{align}
   T^{\sharp}&(\alpha; N) \nonumber \\
   &= \frac{N^{\omega}}{N^{1/k}}(\log N)\, g(\alpha;N) - \frac{(N/h)^{\omega}}{(N/h)^{1/k}}(\log N/h)\, g(\alpha; N/h) - \int_{N/h}^{N} g(\alpha;t)\,\mathrm{d}\bigg(\frac{t^{\omega}}{t^{1/k}}\log t\bigg) \nonumber \\
   &\ll N^{\omega-\frac{1}{k}}(\log N)\, \frac{N^{1/k}}{(\log N)^{4^{1-k}C- \sigma}} = \frac{N^{\omega}}{(\log N)^{4^{1-k} C - \sigma - 1}}. \label{bdsupP}
  \end{align}
  Since $h\geq 2H_0+1$, it follows from Lemma \ref{cntsols} that
  \begin{align*}
   \int_{\mathfrak{m}} |T^{\sharp}(\alpha;N)|^{h}\,\mathrm{d}\alpha &\leq \bigg(\sup_{\alpha\in\mathfrak{m}} |T^{\sharp}(\alpha; N)|^{h-2H_0} \bigg)\,\int_{0}^{1} |T^{\sharp}(\alpha;N)|^{2H_0}\,\mathrm{d}\alpha \\
   &\ll \bigg(\frac{N^{\omega}}{(\log N)^{4^{1-k}C - \sigma -1}}\bigg)^{h-2H_0} N^{2H_0\omega - 1}(\log N)^{M+2H_0} \\
   &= \frac{N^{h\omega-1}}{(\log N)^{R}},
  \end{align*}
  where $R = R(C) = (h- 2H_0)(4^{1-k}C-\sigma-1) - (M + 2H_0) \to \infty$ as $C\to\infty$.
 \end{proof} 
 
 As in the Waring case, we only needed to assume $\omega > 0$ up to here. Only in the next (and final) argument we will use that $\omega \geq 1/h - \delta$ for some $0 \leq \delta < (h-2H_0)/(2h(h-1)H_0)$.

 \begin{proof}[Proof of Proposition \ref{mARC}]
  Assume first that $\omega\geq 1/h$, so that $\max\{h\omega-1,\,0\} = h\omega-1$. Then, by Lemmas \ref{cntsols2} and \ref{bdH}, H\"older's inequality yields
  \begin{align*}
   \sum_{j=1}^{h} \bigg|\int_{\mathfrak{m}} T^{\sharp}(\alpha; N)^{j}\,T(\alpha;N)^{h-j}\,& e(-N\alpha) \,\mathrm{d}\alpha\bigg| \\
   &\leq \sum_{j=1}^{h} \bigg(\int_{\mathfrak{m}} |T^{\sharp}(\alpha;N)|^h\,\mathrm{d}\alpha \bigg)^{j/h}\,\bigg(\int_{0}^{1} |T(\alpha;N)|^{h}\,\mathrm{d}\alpha\bigg)^{1-j/h} \\
   &\ll_R \sum_{j=1}^{h} \bigg(\frac{N^{h\omega-1}}{(\log N)^{R}}\bigg)^{j/h}\, \big(N^{h\omega-1}(\log N)^{h+ M + 2H_0} \big)^{1 - j/h} \\
   &\ll \frac{N^{h\omega-1}}{(\log N)^{R'}},
  \end{align*}
  where $R' := R/h - (h+M+2H_0)(h-1)/h$. Since $R' = R'(C) \to \infty$ as $C\to\infty$, this proves the first part.
  
  Now, as in the proof of Proposition \ref{minorwar}, suppose that $1/h-\delta \leq \omega < 1/h$, and put $\theta := 1-(h-1)\omega$. Our hypothesis on $\delta$ ensures that $\theta < \frac{1}{2H_0}$. For $1\leq j\leq h$, another application of H\"older's inequality gives
  \begin{equation}\label{minarc-holderP}
   \bigg|\int_{\mathfrak{m}} T^{\sharp}(\alpha;N)^{j}\,T(\alpha;N)^{h-j}\,e(-N\alpha)\,\mathrm{d}\alpha\bigg| \ll \bigg(\sup_{\alpha\in\mathfrak{m}} |T^{\sharp}(\alpha;N)|\bigg)^{1-2H_0\theta}\, \Upsilon_1^{\theta}\, \Upsilon_2^{(j-1)\omega}\, \Upsilon_3^{(h-j)\omega},
  \end{equation}
  where
  \[ \Upsilon_1:=\int_{0}^{1} |T^{\sharp}(\alpha;N)|^{2H_0}\,\mathrm{d}\alpha,\quad
     \Upsilon_2:=\int_{0}^{1} |T^{\sharp}(\alpha;N)|^{1/\omega}\,\mathrm{d}\alpha,\quad
     \Upsilon_3:=\int_{0}^{1} |T(\alpha;N)|^{1/\omega}\,\mathrm{d}\alpha. \]
  By \eqref{bdsupP} in Lemma \ref{bdH} we have $\sup_{\alpha\in\mathfrak{m}} |T^{\sharp}(\alpha;N)| \ll N^{\omega}/(\log N)^{R(C)}$ with $R(C) \to\infty$ as $C\to\infty$, and by Lemmas \ref{cntsols} and \ref{cntsols2} we have $\Upsilon_1 \ll N^{2H_0\omega-1}(\log N)^{M_1}$ and $\Upsilon_2, \Upsilon_3 \ll (\log N)^{M_2}$ for some constants $M_1=M_1(k) > 0$ and $M_2 = M_2(k,h,\delta)>0$ (since $(1/\omega)\omega-1=0$ and $1/\omega > h > 2H_0$ in the range considered). Substituting these into \eqref{minarc-holderP}, we obtain
  \[ \bigg|\int_{\mathfrak{m}} T^{\sharp}(\alpha;N)^{j}\,T(\alpha;N)^{h-j}\,e(-N\alpha)\,\mathrm{d}\alpha\bigg| \ll \frac{N^{h\omega-1}}{(\log N)^{R'}}, \]
  with $R' = R(C)(1-2H_0\theta) - M_1\theta - (h-1)\omega M_2$. Since $1-2H_0\theta > 0$, we have $R'=R'(C) \to \infty$ as $C\to\infty$, and this concludes the proof.
 \end{proof}

%%%%%%%%%%%%%%%%%%%%%%%
\addtocontents{toc}{\protect\setcounter{tocdepth}{0}}
 \section*{Acknowledgements}
  The author thanks Andrew Granville for helpful advice during the preparation of this article, Cihan Sabuncu for insightful discussions related to this work, and Trevor Wooley for helpful remarks and references concerning the Vinogradov mean value theorem input.
 
\addtocontents{toc}{\protect\setcounter{tocdepth}{1}}
%%%%%%%%%%%%%%%%%%%%%%%

% ----------------------------------------------------------------
\bibliographystyle{amsplain}
\bibliography{$HOME/Academie/Recherche/_latex/bibliotheca}%
\end{document}